\documentclass[11pt]{article}

\def\0{\bf 0}
\def\1{\bf 1}

\usepackage{amsmath, amsfonts, amssymb, amsthm}

\newcommand{\rrvert}{\vert}

\newcommand{\llvert}{\vert}

\newtheorem{theorem}{Theorem}[section]
\newtheorem{example}{Example}[section]

\newtheorem{definition}{Definition}[section]
\newtheorem{lemma}{Lemma}[section]

\newtheorem{remark}{Remark}[section]
\newtheorem{corollary}{Corollary}[section]

\begin{document}

\title{\bf Variable selection with Hamming loss}
\author{Butucea, C.$^{1,2}$, Ndaoud, M.$^{2}$, Stepanova, N.A.$^{3}$, and Tsybakov, A.B.$^{2}$\\
{\small $^1$ Universit\'e Paris-Est Marne-la-Vall\'ee, LAMA(UMR 8050), UPEM, UPEC, CNRS,}\\ {\small F-77454, Marne-la-Vall\'ee, France} \\
{\small $^2$ CREST, ENSAE, Universit\'e Paris-Saclay.
5, ave. Henry le Chatelier,} \\ 
{\small 91120 Palaiseau Cedex, France}\\
{\small $^3$ School of Mathematics and Statistics, Carleton University,
Ottawa, Ontario, K1S 5B6 Canada}
}

\maketitle


\begin{abstract}
We derive nonasymptotic bounds for the minimax risk of variable
selection under expected Hamming loss in the Gaussian mean model in
$\mathbb{R}^d$ for classes of at most $s$-sparse vectors separated from 0 by a
constant $a>0$. In some cases, we get exact expressions for the
nonasymptotic minimax risk as a function of $d,s,a$ and find explicitly
the minimax selectors. These results are extended to dependent or
non-Gaussian observations and to the problem of crowdsourcing.
Analogous conclusions are obtained for the probability of wrong
recovery of the sparsity pattern. As corollaries, we derive necessary
and sufficient conditions for such asymptotic properties as almost full
recovery and exact recovery. Moreover, we propose data-driven selectors
that provide almost full and exact recovery adaptively to the
parameters of the classes.\end{abstract}


\section{Introduction}
In recent years, the problem of variable selection in high-dimensional
regression models has been
extensively studied from the theoretical and computational viewpoints.
In making effective high-dimensional inference, sparsity plays a key role.
With regard to variable selection in sparse high-dimensional regression,
the Lasso, Dantzig selector, other penalized techniques as well as
marginal regression were analyzed in detail; see, for example,
\cite{MB2006,ZhaoYu2006,Wainwright2009,Lounici2008,WR2009,Zhang2010,MB2010,GJW2012,JJ2012} and the references cited therein.
Several other recent papers deal with sparse variable selection in
nonparametric regression; see, for example, \cite
{LW2008,BL2008,Dalalyan,IngStep2014,ButStep2015}.

In this paper, we study the problem of variable selection in the
Gaussian sequence model
%
\begin{equation}
\label{vectormodel} X_j = \theta_j + \sigma
\xi_j, \qquad j =1,\dots, d,
\end{equation}
where $\xi_1,\dots,\xi_d$ are i.i.d. standard Gaussian random
variables, $\sigma>0$ is the noise level, and $\theta= (\theta
_1,\dots,\theta_d)$ is {an} unknown vector of parameters to be
estimated. We assume that $\theta$ is $(s,a)$-\textit{sparse}, which is
understood in the sense that $\theta$ belongs to one of the following sets:
\begin{eqnarray*}
\Theta_d(s,a) & = & \bigl\{\theta\in\mathbb{R}^d: \mbox{
there exists a set } S\subseteq\{1,\dots,d\} \mbox{ with at most } s \mbox{ elements }
\\
&&  \mbox{such that } \llvert \theta_j \rrvert \geq a \mbox{
for all } j \in S, \mbox{ and } \theta_j=0 \mbox{ for all } j \notin
S \bigr\}
\end{eqnarray*}
or
\begin{eqnarray*}
\Theta_d^+(s,a) & = & \bigl\{\theta\in\mathbb{R}^d:
\mbox{ there exists a set } S\subseteq\{1,\dots,d\} \mbox{ with at most } s \mbox{
elements }
\\
&&  \mbox{such that } \theta_j \geq a \mbox{ for all } j \in
S, \mbox{ and } \theta_j=0 \mbox{ for all } j \notin S \bigr\}.
\end{eqnarray*}
Here, $a>0$ and $s\in\{1,\dots,d\}$ are given constants.

We study the problem of selecting
the {relevant} components of $\theta$, that is, of estimating the vector
\[
\eta= \eta(\theta) = \bigl(I(\theta_j \ne0 )\bigr)_{j=1,\dots,d},
\]
where $I(\cdot)$ is the indicator function.
As estimators of $\eta$, we consider any measurable functions
$\hat\eta=\hat\eta(X_1,\dots,X_n)$ of $(X_1,\dots,X_n)$
taking values in $\{0,1\}^d$. Such estimators will be called \textit{selectors}.
We characterize the loss of {a} selector $\hat\eta$ as {an}
estimator of $\eta$ by the Hamming distance between $\hat\eta$
and $\eta$, that is, by the number of positions at which $\hat
\eta$ and $\eta$ differ:
\[
\llvert \hat\eta-\eta \rrvert \triangleq\sum_{j=1}^d
\llvert \hat\eta_j-\eta_j \rrvert = \sum
_{j=1}^d I(\hat\eta_j\ne
\eta_j).
\]
Here, $\hat\eta_j$ and $ \eta_j=\eta_j(\theta)$ are the $j$th
components of $\hat\eta$ and $ \eta= \eta(\theta)$, respectively.
The expected Hamming loss of a selector $\hat\eta$ is defined as
$\mathbf{E}_\theta \llvert \hat\eta- \eta \rrvert $,
where $\mathbf{E}
_\theta$ denotes the expectation with respect to the distribution
$\mathbf{P}
_\theta$ of $(X_1,\dots,X_n)$ satisfying \eqref{vectormodel}.\vspace*{2pt}
Another well-known risk measure is the probability of wrong recovery
$\mathbf{P}_\theta(\hat S \ne S(\theta))$, where $\hat S= \{j:
\hat\eta_j=1\}$ and $S(\theta)=\{ j:   \eta_j(\theta)=1\}$.
It can be viewed as the Hamming distance with an indicator loss and is
related to the expected Hamming loss as follows:
%
\begin{equation}
\label{risks} \mathbf{P}_\theta\bigl(\hat S \ne S(\theta)\bigr) =
\mathbf{P}_\theta \bigl( \llvert \hat\eta- \eta \rrvert \ge1\bigr)
\le\mathbf{E}_\theta \llvert \hat\eta- \eta \rrvert .
\end{equation}
In view of the last inequality, bounding the expected Hamming loss
provides a stronger result than bounding the probability of wrong recovery.

Most of the literature on variable selection in high dimensions focuses
on the recovery of the sparsity pattern, that is, on constructing
selectors such that the probability $\mathbf{P}_\theta(\hat S \ne
S(\theta))$ is close to 0 in some asymptotic sense (see, e.g., \cite
{MB2006,ZhaoYu2006,Wainwright2009,Lounici2008,WR2009,Zhang2010,MB2010}). These
papers consider high-dimensional linear regression settings with
deterministic or random covariates. In particular, for the sequence
model \eqref{vectormodel}, one gets that if $a>C\sigma\sqrt{\log d}$
for some $C>0$ large enough, then there exist selectors such that
$\mathbf{P}
_\theta(\hat S \ne S(\theta))$ tends to 0, while this is not the
case if $a<c\sigma\sqrt{\log d}$ for some $c>0$ small enough. More
insight into variable selection was provided in \cite{GJW2012,JJ2012}
by considering\vspace*{1pt} a Hamming risk close to the one we have defined above.
Assuming that $s\sim d^{1-\beta}$ for some $\beta\in(0,1)$, the
papers \cite{GJW2012,JJ2012} establish an asymptotic in $d$ ``phase
diagram'' that partitions the parameter space into three regions called the
exact recovery, almost full recovery, and no recovery regions. This is
done in a Bayesian setup for the linear regression model with i.i.d.
Gaussian covariates and random $\theta$. Note also that in \cite
{GJW2012,JJ2012} the knowledge of $\beta$ is required to construct the
selectors, so that in this sense the methods are not adaptive. The
selectors are of the form
$\hat\eta_j = I( \llvert  X_j \rrvert \geq t)$ with threshold
$t=\tau(\beta)\sigma\sqrt{\log d}$ for some function $\tau(\cdot
)>0$. More recently, these asymptotic results were extended to a
combined minimax--Bayes Hamming risk on a certain class of vectors
$\theta$ in \cite{JZZ2014}.

The present paper makes further steps in the analysis of variable
selection with a Hamming loss initiated in \cite{GJW2012,JJ2012}.
Unlike \cite{GJW2012,JJ2012}, we study the sequence model \eqref
{vectormodel} rather than Gaussian regression and analyze the behavior
of the minimax risk rather than that of the Bayes risk with a specific
prior. Furthermore, we consider not only $s\sim d^{1-\beta}$ but
general $s$ and derive nonasymptotic results that are valid for any
sample size. Remarkably, we get an exact expression for the
nonasymptotic minimax risk of separable (coordinate-wise) selectors and find explicitly the separable minimax selectors.
Finally, we construct data-driven selectors that are simultaneously
adaptive to the parameters $a$ and $s$.

Specifically, we consider the minimax risk
%
\begin{equation}
\label{minimaxrisk} \inf_{\tilde\eta} \sup_{\theta\in\Theta}
\frac{1}s \mathbf{E} _\theta \llvert \tilde\eta- \eta \rrvert
\end{equation}
for $\Theta= \Theta_d(s,a)$ and $\Theta= \Theta_d^+(s,a)$, where
$\inf_{\tilde\eta}$ denotes the infimum over all
selectors $\tilde\eta$.
In Section~\ref{sec:minimax}, for both classes $\Theta= \Theta
_d(s,a)$ and $\Theta= \Theta_d^+(s,a)$ we find the upper and lower bounds of
the minimax risks and derive minimax selectors for any fixed $d, s,
a>0$ such that $s < d $. For $\Theta= \Theta_d(s,a)$, we also propose
another selector attaining the lower bound risk up to the factor 2.
Interestingly, the thresholds that correspond to the minimax optimal
selectors do not have the classical form $A\sigma\sqrt{\log d}$ for
some $A>0$; the optimal threshold is a function of $a$ and $s$.
Analogous minimax results are obtained for the risk measured by the
probability of wrong recovery $\mathbf{P}_\theta(\hat S \ne
S(\theta
))$. Section~\ref{sec:general} considers extensions of the
nonasymptotic minimax theorems of Section~\ref{sec:minimax} to
settings with non-Gaussian or dependent observations.
In Section~\ref{sec:asymp}, as asymptotic corollaries of these
results, we establish sharp conditions under which exact and almost
full recovery are achievable. Section~\ref{sec:adapt} is devoted to
the construction of adaptive selectors that achieve almost full and
exact recovery without the knowledge of the parameters $a$ and $s$.
Most of the proofs are given in the Appendix.

Finally, note that quite recently several papers have studied the
expected Hamming loss in other problems of variable selection.
Asymptotic behavior of the minimax risk analogous to \eqref
{minimaxrisk} for classes $\Theta$ different from the sparsity classes
that we consider here was analyzed in \cite{ButStep2015} and
without the normalizing factor $1/s$ in \cite{IngStep2014}. Oracle
inequalities for Hamming risks in the problem of multiple
classification under sparsity constraints are established in \cite
{NR2012}. The paper \cite{ZhangZhou2015} introduces an asymptotically
minimax approach based on the Hamming loss in the problem of community
detection in networks.

\section{Nonasymptotic minimax selectors}\label{sec:minimax}

In what follows, we assume that \mbox{$s < d$}. We first consider minimax
variable selection for the class
$\Theta_d^+(s,a)$. For this class, we will use a selector $\hat\eta
^+$ with the components
%
\begin{equation}
\label{selector+} \hat\eta_j^+ = I(X_j\geq t), \qquad j =1,
\dots,d,
\end{equation}
where the threshold is defined by
%
\begin{equation}
\label{threshold} t = \frac{a}2 + \frac{\sigma^2} a \log \biggl(
\frac{d}s - 1 \biggr).
\end{equation}
Set
\[
\Psi_+(d,s,a)= \biggl( \frac{d}s - 1 \biggr) \Phi \biggl( -
\frac{a}{
2\sigma} - \frac{\sigma} a \log \biggl( \frac{d}s - 1 \biggr)
\biggr) +\Phi \biggl( - \frac{a} {2\sigma} + \frac{\sigma} a \log \biggl(
\frac{d}s - 1 \biggr) \biggr),
\]
where $\Phi(\cdot)$ denotes the standard Gaussian cumulative
distribution function.

\begin{theorem}\label{t2} For any $a>0$ and $s \leq d/2$, the selector
$\hat\eta^+$ in \eqref{selector+}
with the threshold $t$ defined in (\ref{threshold}) satisfies
%
\begin{equation}
\label{t2:eq1} \sup_{\theta\in\Theta_d^+(s,a)} \frac{1}s
\mathbf{E}_\theta \bigl\llvert \hat\eta^+ - \eta \bigr\rrvert \leq
\Psi_+(d,s,a).
\end{equation}
\end{theorem}
The proof is given in the Appendix.

A selector $\tilde\eta=(\tilde\eta_1,\dots,\tilde\eta
_d)$ will be called \textit{separable} if its $j$th component $\tilde
\eta_j$ depends only on $X_j$ for all $j=1,\dots,d$. We denote by
$\mathcal T$ the set of all separable selectors.

The next theorem gives a lower bound on the minimax risk showing that
the upper bound in Theorem~\ref{t2} is tight over separable selectors.

\begin{theorem}\label{thm:lb} For any $a>0$ and $s < d$, we have
\[
\inf_{\tilde\eta \in \mathcal T} \sup_{\theta\in\Theta_d^+(s,a)} \frac{1}s
\mathbf{E}_\theta \llvert \tilde\eta- \eta \rrvert \ge\Psi_+(d,s,a),
\]
where $\inf_{\tilde\eta \in \mathcal T}$ denotes the infimum over all separable selectors
$\tilde\eta$. Moreover, for any $s'$ in $(0,s]$, we have
\[
\inf_{\tilde\eta } \sup_{\theta\in\Theta_d^+(s,a)} \frac{1}s
\mathbf{E}_\theta \llvert \tilde\eta- \eta \rrvert \ge \frac{s'}{s} \Psi_+(d,s,a) - \frac{4s'}{s} \exp \left( - \frac{(s-s')^2}{2s}\right),
\]
where $\inf_{\tilde\eta }$ denotes the infimum over all selectors
$\tilde\eta$.
\end{theorem}

The proof of the first inequality of Theorem~\ref{thm:lb} is given in the Appendix, while the proof of the second inequality is given in the 
Supplementary material \cite{supp}.

As a straightforward corollary of Theorems \ref{t2} and \ref{thm:lb},
we obtain that the estimator $\hat\eta^+$ is minimax among the separable selectors in the exact
sense for the class $\Theta_d^+(s,a)$ and the minimax risk satisfies
%
\begin{equation}
\label{eq:minimax:central} \inf_{\tilde\eta \in \mathcal T} \sup_{\theta\in\Theta_d^+(s,a)}
\frac{1}s \mathbf{E}_\theta \llvert \tilde\eta- \eta \rrvert =
\sup_{\theta\in\Theta_d^+(s,a)} \frac{1}s \mathbf{E}_\theta \bigl
\llvert \hat\eta^+ - \eta \bigr\rrvert = \Psi _+(d,s,a).
\end{equation}
Remarkably, this holds under no assumptions on $d,s,a$ except for, of
course, some minimal conditions under which the problem ever makes
sense: $a>0$ and $s \leq d/2$. Analogous non-asymptotic minimax result is
valid for the class
\begin{eqnarray*}
\Theta_d^-(s,a) & = & \bigl\{\theta\in\mathbb{R}^d:
\mbox{ there exists a set } S\subseteq\{1,\dots,d\} \mbox{ with at most } s \mbox{
elements }
\\
&&  \mbox{such that } \theta_j \leq- a \mbox{ for all } j
\in S, \mbox{ and } \theta_j=0 \mbox{ for all } j \notin S \bigr\}.
\end{eqnarray*}
We omit details here.

Next, consider the class $\Theta_d(s,a)$. A direct analog of $\hat
\eta^+$ for $\Theta_d(s,a)$ is a selector
$\hat\eta$ with the components
%
\begin{equation}
\label{selector} \hat\eta_j = I\bigl( \llvert X_j \rrvert
\geq t\bigr), \qquad j =1,\dots,d,
\end{equation}
where the threshold $t$ is defined in \eqref{threshold}.
Set
\[
\begin{aligned}
\Psi(d,s,a)&= \biggl( \frac{d}s - 1 \biggr) \Phi \biggl( -
\frac{a}{2\sigma} - \frac{\sigma} a \log \biggl( \frac{d}s - 1 \biggr)
\biggr) \\
&\quad {}+\Phi \biggl( - \biggl(\frac{a}{2\sigma} - \frac{\sigma} a \log \biggl(
\frac{d}s - 1 \biggr) \biggr)_+ \biggr),
\end{aligned}
\]
where $x_+=\max(x,0)$. Note that
%
\begin{equation}
\label{eq_psi} \Psi(d,s,a) \le\Psi_+(d,s,a).
\end{equation}
We have the following bound.

%

\begin{theorem}\label{nonasymptotic}
For any $a>0$ and $s \leq d/2$, the selector $\hat\eta$ in (\ref
{selector}) with the threshold $t$ defined in (\ref{threshold}) satisfies
%
\begin{equation}
\label{t1:eq1} \sup_{\theta\in\Theta_d(s,a)} \frac{1}s
\mathbf{E}_\theta \llvert \hat \eta- \eta \rrvert \leq2 \Psi(d,s,a).
\end{equation}
\end{theorem}
The proof is given in the Appendix.

For the minimax risk on the class $\Theta_d(s,a)$, we have the
following corollary, which is an immediate consequence of Theorems \ref
{thm:lb}, \ref{nonasymptotic} and
inequality \eqref{eq_psi}.

\begin{corollary}\label{cor:factor2}
For any $a>0$ and $s \leq d/2$, the selector $\hat\eta$ in (\ref
{selector}) with the threshold $t$ defined in (\ref{threshold}) satisfies
%
\begin{equation}
\label{cor1:eq1} \sup_{\theta\in\Theta_d(s,a)} \mathbf{E}_\theta \llvert
\hat \eta- \eta \rrvert \leq2 \inf_{\tilde\eta \in \mathcal{T}} \sup
_{\theta\in
\Theta_d(s,a)} \mathbf{E}_\theta \llvert \tilde\eta- \eta
\rrvert .
\end{equation}
\end{corollary}

Thus, the risk of the thresholding estimator (\ref{selector}) cannot
be greater than
the minimax risk of separable selectors over the class $\Theta_d(s,a)$ multiplied by 2.

We turn now to exact minimax variable selection over the class $\Theta
_d(s,a)$. Consider a selector $\bar\eta=(\bar\eta_1,\dots,\bar\eta_d)$ with the components
%
\begin{equation}
\label{selector2} \bar\eta_j = I \biggl(\log \biggl(\cosh \biggl(
\frac{a
X_j}{\sigma^2} \biggr) \biggr) \geq t \biggr), \qquad j =1,\dots,d,
\end{equation}
where the threshold is defined by
%
\begin{equation}
\label{threshold2} t = \frac{a^2}{2 \sigma^2} + \log \biggl( \frac{d}s - 1
\biggr).
\end{equation}
Set
\begin{eqnarray*}
\bar\Psi(d,s,a)&=& \biggl(\frac{d}s - 1 \biggr) \mathbf{P} \biggl(
e^{-
\frac{a^2}{2\sigma^2} }\cosh \biggl(\frac{a\xi}{\sigma} \biggr) \geq\frac{d}s - 1
\biggr)
\\
&&{}+ \mathbf{P} \biggl( e^{- \frac{a^2}{2\sigma^2} }\cosh \biggl(\frac
{a\xi }{\sigma}+
\frac{a^2}{\sigma^2} \biggr) < \frac{d}s - 1 \biggr),
\end{eqnarray*}
where $\xi$ denotes a standard Gaussian random variable.
Our aim is to show that
$\bar\Psi(d,s,a)$ is the minimax risk of variable selection
under the Hamming loss over the class $\Theta_d(s,a)$ and that it is
nearly achieved by the selector in \eqref{selector2}. We first prove that
$\bar\Psi(d,s,a) d/(d-s)$ is an upper bound on the maximum risk of the selector \eqref{selector2}.

\begin{theorem}\label{t1}
For any $a>0$ and $s < d$, the selector $\bar\eta$ in (\ref
{selector2}) with the threshold $t$ defined in (\ref{threshold2}) satisfies
\begin{equation*}
\sup_{\theta\in\Theta_d(s,a)} \frac{1}s \mathbf{E}_\theta
\llvert \bar\eta- \eta \rrvert \leq\bar\Psi(d,s,a) \frac d{d-s}.
\end{equation*}
\end{theorem}

The next theorem establishes the lower bound over separable selectors on the minimax risk associated to the upper bound in Theorem~\ref{t1}.

\begin{theorem}\label{thm:lb2} For any $a>0$ and $s < d$, we have
\[
\inf_{\tilde\eta \in \mathcal T} \sup_{\theta\in\Theta_d(s,a)} \frac{1}s
\mathbf{E} _\theta \llvert \tilde\eta- \eta \rrvert \ge\bar\Psi(d,s,a),
\]
where $\inf_{\tilde\eta \in \mathcal T}$ denotes the infimum over all separable selectors
$\tilde\eta$.
\end{theorem}

%
%
%
%

%

Finally, we show how the above nonasymptotic minimax results can be
extended to the probability of wrong recovery. For any selector
$\tilde\eta$, we denote by $ S_{\tilde\eta}$ the selected
set of indices: $S_{\tilde\eta} = \{j:  \tilde\eta_j=1\}$.

\begin{theorem}\label{cor:recovery_pattern}
For any $a>0$ and $s \leq d/2$, the selectors $\hat\eta$ in \eqref
{selector} and $\hat\eta^+$ in \eqref{selector+} with the threshold
$t$ defined in \eqref{threshold}, and the selector $\bar\eta$
in (\ref{selector2}) with the threshold $t$ defined in (\ref
{threshold2}) satisfy
%
\begin{align}
\label{cor:recovery_pattern:eq1} \sup_{\theta\in\Theta_d^+(s,a)} \mathbf{P}_\theta\bigl(
S_{\hat\eta
^+} \ne S(\theta)\bigr) &\le s\Psi_+(d,s,a),
\\
\label{cor:recovery_pattern:eqbar2} \sup_{\theta\in\Theta_d(s,a)} \mathbf{P}_\theta\bigl(
S_{\bar\eta} \ne S(\theta)\bigr) &\le s \bar\Psi(d,s,a) \frac d{d-s}
\end{align}
and
%
\begin{equation}
\label{cor:recovery_pattern:eq2} \sup_{\theta\in\Theta_d(s,a)} \mathbf{P}_\theta\bigl(
S_{\hat\eta} \ne S(\theta)\bigr) \le2s\Psi(d,s,a).
\end{equation}
Furthermore,
%
\begin{equation}
\label{cor:recovery_pattern:eq3} \inf_{\tilde\eta\in\mathcal T} \sup_{\theta\in\Theta
_d^+(s,a)}
\mathbf{P}_\theta\bigl(S_{\tilde\eta} \ne S(\theta)\bigr) \ge
\frac
{s\Psi_+(d,s,a)}{1+s\Psi_+(d,s,a)}
\end{equation}
and
%
\begin{equation}
\label{cor:recovery_pattern:eqbar3} \inf_{\tilde\eta\in\mathcal T} \sup_{\theta\in\Theta
_d(s,a)}
\mathbf{P}_\theta\bigl(S_{\tilde\eta} \ne S(\theta)\bigr) \ge
\frac{s
\bar\Psi(d,s,a)}{1+s \bar\Psi(d,s,a)}.
\end{equation}
\end{theorem}
The proof is given in the Appendix.

Although Theorem~\ref{cor:recovery_pattern} does not provide the exact
minimax solution, it implies sharp minimaxity in asymptotic sense.
Indeed, an interesting case is when the minimax risk in Theorem~\ref
{cor:recovery_pattern} goes to 0 as $d\to\infty$. Assuming that $s$
and $a$ are functions of $d$, this corresponds to $s\Psi_+(d,s,a)\to
0$ as $d\to\infty$. In this natural asymptotic setup, the upper and
lower bounds of Theorem~\ref{cor:recovery_pattern} for the class $
\Theta_d^+(s,a)$ are sharp. The same for the class $ \Theta_d(s,a)$, if $s$ and $a$ are such that $s \bar\Psi(d,s,a) \to 0$ and that $s/d\to 0$. We discuss this
issue in Section~\ref{sec:asymp}; cf. Theorem~\ref{th:sharp:asymp}.

\begin{remark} Papers \cite{GJW2012,JJ2012,JZZ2014} use a
different Hamming loss defined in terms of vectors of signs. In our setting,
this would mean considering not $ \llvert \hat\eta- \eta \rrvert $ but the following loss:
$
\sum_{j=1}^d I( \operatorname{sign}(\hat\theta_j)\ne\operatorname{sign}(\theta_j))$,
where $\hat\theta_j$ is an estimator of $\theta_j$ and $\operatorname{sign}(x)=I(x>0)-I(x<0)$. Theorems of this section are easily adapted to
such a loss, but in this case the corresponding expressions for the
nonasymptotic risk contain additional terms and we do not obtain exact
minimax solutions as above. On the other hand, these additional terms
are smaller than $\Psi(d,s,a)$ and $\Psi_+(d,s,a)$, and in the
asymptotic analysis, such as the one performed in Sections~\ref{sec:asymp} and \ref{sec:adapt}, can often be
neglected.
Thus, in many cases, one gets the same asymptotic results for both
losses. We do not discuss this issue in more detail here.
\end{remark}

\section{Generalizations and extensions}\label{sec:general}
Before proceeding to asymptotic corollaries, we discuss some
generalizations and extensions of the nonasymptotic results of Section~\ref{sec:minimax}.
\subsection{Dependent observations}
It is easy to see that Theorems \ref{t2} and \ref{nonasymptotic} do
not use any information on the dependence between the observations, and
thus remain valid for dependent $X_j$. Furthermore, a minimax
optimality property within the class of separable selectors holds under dependence as well. To be specific,
denote by $\mathcal{N}_d(\theta,\Sigma)$ the $d$-dimensional
Gaussian distribution with mean $\theta$ and covariance matrix $\Sigma
$. Assume that the distribution $\mathbf{P}$ of $(X_1,\dots,X_d)$
belongs to
the class
\[
{\mathcal P}_d^+\bigl(s,a, \sigma^2\bigr) = \bigl\{
\mathcal{N}_d(\theta,\Sigma): \theta\in\Theta_d^+(s,a),
\sigma_{ii}=\sigma^2, \mbox{ for all }i = 1,\dots,d\bigr\},
\]
where we denote by $\sigma_{ii}$ the diagonal entries of $\Sigma$.
Note that, for distributions in this class, $\Sigma$ can be any
covariance matrix with constant diagonal elements.
%
\begin{theorem}\label{thm:dependence} For any $a>0$ and $s \leq d/2$, and
for the selector $\hat\eta^+$ in \eqref{selector+}
with the threshold $t$ defined in (\ref{threshold}) we have
\[
\inf_{\tilde\eta \in \mathcal{T}} \sup_{\mathbf{P}\in{\mathcal
P}_d^+(s,a,\sigma
^2)} \mathbf{E}_\mathbf{P}
\llvert \tilde\eta- \eta \rrvert = \sup_{\mathbf{P}\in{\mathcal P}_d^+(s,a,\sigma^2)}
\mathbf{E}_\mathbf{P} \bigl\llvert \hat\eta^+ - \eta \bigr\rrvert = s
\Psi_+(d,s,a),
\]
where $\inf_{\tilde\eta \in \mathcal{T}}$ denotes the infimum over all separable selectors $\tilde\eta$, and $\mathbf{E}_\mathbf{P}$ denotes the expectation with respect to $\mathbf{P}$.
Moreover, for any $s'$ in $(0,s]$, we have
\[
\inf_{\tilde\eta } \sup_{\mathbf{P}\in{\mathcal
P}_d^+(s,a,\sigma
^2)} \mathbf{E}_\mathbf{P}
\llvert \tilde\eta- \eta \rrvert \ge {s'} \Psi_+(d,s,a) - 4 {s'} \exp \left( - \frac{(s-s')^2}{2s}\right),
\]
where $\inf_{\tilde\eta}$ denotes the infimum over all selectors
$\tilde\eta$.
\end{theorem}
\begin{proof} The upper bound $\Psi_+(d,s,a)$ on the minimax risk follows
from the fact that the proofs of Theorems \ref{t2} and \ref
{nonasymptotic} are not affected by the dependence. Indeed, both the
selector and the Hamming loss proceed coordinate-wisely. The lower
bound on the minimax risk follows from Theorem~\ref
{thm:lb} after observing that the maximum over ${\mathcal
P}_d^+(s,a,\sigma^2)$ is greater than the maximum over the subfamily
of Gaussian vectors with independent entries $\{\mathcal{N}_d(\theta,
\sigma^2 I_d):\theta\in\Theta_d^+(s,a)\}$, where $I_d$ is the
$d\times d$ identity matrix.
\end{proof}
An interesting consequence of Theorem~\ref{thm:dependence} and of
\eqref{eq:minimax:central} is that the model with independent $X_j$ is
the least favorable model, in the exact nonasymptotic sense, for the
problem of variable selection with Hamming loss on the class of vectors
$\Theta_d^+(s,a)$.

This fact was also noticed and discussed in \cite{HallJin2010} for the
detection problem. That paper considers the Gaussian model with
covariance matrix $\Sigma$ that is not necessarily a diagonal matrix.
It is shown that faster detection rates are achieved in the case of
dependent observations (under some assumptions) than in the case of
independent data. It would be interesting to extend these results to
the variable selection problem in hand.

\subsection{Non-Gaussian models}
As a building block for extension to non-Gaussian observations, we
first consider the following simple model. We observe independent
random variables $X_1,\dots,X_d$ with values in a measurable space
$({\mathcal X},{\mathcal U})$ such that at most $s$ among them are distributed
according to the probability distribution $P_1$ and the other are distributed according to the probability distribution
$P
_0$. We assume that $P_0\ne P_1$. Let $f_0$ and
$f_1$ be
densities of $P_0$ and $P_1$ with respect to some dominating
measure. Denote by $\eta=(\eta_1,\dots,\eta_d)$ the vector such
that $\eta_j=1 $ if the distribution of $X_j$ is $P_1$ and
$\eta
_j=0$ if it is $P_0$. Define $\Theta_d(s)$ as the set of all vectors $\eta \in \{0,1\}^{d}$ with at most $s$ non-zero components. For any fixed $\eta$, we denote by $\mathbf{E}_{\eta}$ the expectation with respect to the distribution of $(X_{1},\dots,X_{d})$. Consider the selector $\hat\eta=(\hat\eta
_1,\dots,\hat\eta_d)$, where
%
\begin{equation}
\label{def:newsel} \hat\eta_j = I \biggl(s f_1(X_j)
\geq (d-s ) f_0(X_j) \biggr), \qquad j=1,\dots,d.
\end{equation}

\begin{theorem}\label{thm:fixed}
For any $s<d$, the selector $\hat\eta$ in (\ref{def:newsel}) satisfies
%
\begin{equation*}
\sup_{\eta\in\Theta_d(s)} \mathbf{E}_{\eta} \frac 1s \llvert \hat\eta - \eta
\rrvert \leq \Psi(d,s) \frac d{d-s},
\end{equation*}
and, for any $s'$ in $(0,s]$,
\begin{equation}
\label{thm:fixed:eq1} \inf_{\tilde\eta} \sup_{\eta\in\Theta_d(s )} \frac 1s \mathbf{E}_{\eta}
\llvert \tilde\eta- \eta \rrvert \geq \frac{s'}s \Psi(d,s) - \frac{4 s'}s \exp \left(-\frac{(s-s')^2}{2s} \right),
\end{equation}
where $ \inf_{\tilde{\eta}}$ denotes the infimum over all
selectors, and
\begin{equation}
\label{psi}
\begin{aligned}
\Psi = \Psi(d,s) &= P_1 \biggl(s f_1(X_1)
< (d-s ) f_0(X_1)\biggr)\\
&\quad {} + \biggl(\frac{d}s -
1 \biggr) P_0 \biggl( s f_1(X_1)
\geq (d-s ) f_0(X_1) \biggr).
\end{aligned}
\end{equation}
\end{theorem}
%
%
%
The proof is given in the Supplementary material \cite{supp}.

Suppose now that instead of two measures $P_0$ and $P_1$ we have
a parametric family of probability measures $\{\mathbb{P}_a, a\in
\mathcal
{U}\}$ where $\mathcal{U}\subseteq{\mathbb R}$. Let $\mathsf{ f}_a$ be a
density of $\mathbb{P}_a$ with respect to some dominating measure. Recall
that the family $\{\mathsf{ f}_a, a\in\mathcal{U}\}$ is said to have the
Monotone Likelihood Ratio (MLR) property if, for all $a_0,a_1$ in
$\mathcal{U}$ such that $a_0 < a_1$, the log-likelihood ratio $
\log({\mathsf{ f}_{a_1}(X) }/{\mathsf{ f}_{a_0}(X) } )
$
is an increasing function of $X$. In particular, this implies (cf.
\cite{lehmann}, {Lemma~3.4.2}) that $\{\mathsf{ f}_a,   a \in\mathcal
{U}\}$ is a stochastically ordered family, that is,
%
\begin{equation}
\label{order} F_{a} (x) \geq F_{a'} (x) \qquad\mbox{for all
} x\mbox{ if } a < a',
\end{equation}
where $F_a$ is the cumulative distribution function corresponding to
$\mathsf{ f}_a$.
Using these facts, we generalize the nonasymptotic results of the
previous section in two ways. First, we allow for not necessarily
Gaussian distributions and second, instead of the set of parameters
$\Theta_d^+(s,a)$, we consider
the following set with two restrictions:
\begin{eqnarray*}
\Theta_d^+ (s,a_0,a_1) &=& \bigl\{ \theta
\in\mathbb{R}^d: \exists \mbox{ a set }S\subseteq\{1,\dots,d\}
\mbox{ with at most $s$ elements}
\\
&&  \mbox{ such that $\theta_j \geq a_1$ for
all $j \in S$, and $\theta_j \leq a_0$ for all } j
\notin S \bigr\},
\end{eqnarray*}
where $a_0<a_1$. We assume that $X_{j}$ is distributed with density ${\sf f}_{\theta_{j}}$ for $j=1,\dots,d$, and $X_{1},\dots,X_{d}$ are independent. In the next theorem, $\mathbf{E}_{\theta}$ is the expectation with respect to the distribution of such $X_{1},\dots,X_{d}$. In what follows, we use the notation $f_j = \mathsf{f}_{a_j}, j=0,1$.
%
\begin{theorem}\label{thm:st} Let $\{\mathsf{ f}_a,   a \in\mathcal{U}\}
$ be a family with the MLR property, and let $a_0,a_1\in\mathcal{U}$
be such that $a_0 < a_1$.
Set $f_0 = \mathsf{ f}_{a_0}$ and $f_1= \mathsf{ f}_{a_1}$, then, for any $s<d$, the selector $\hat\eta$ in (\ref{def:newsel}) satisfies
\[
\sup_{\theta\in\Theta_d^+(s, a_0,a_1)} \frac 1s \mathbf{E}_\theta \llvert \hat \eta- \eta
\rrvert \leq \Psi(d,s) \frac d{d-s},
\]
and, for any $s'$ in $(0,s]$,
\[
\inf_{\tilde\eta} \sup_{\theta\in
\Theta_d^+(s, a_0,a_1)} \frac 1s
\mathbf{E}_\theta \llvert \tilde\eta- \eta \rrvert \geq \frac{s'}s\Psi(d,s) - \frac{4s'}s \exp \left(-\frac{(s-s')^2}{2s} \right),
\]
where $ \inf_{\tilde{\eta}}$ denotes the infimum over all selectors
and $\Psi$ is given in (\ref{psi}).
\end{theorem}

%
%
%
%
%
The proof is given in the Supplementary Material \cite{supp}.
\begin{example}\label{exmp1}
Let $\mathsf{ f}_a $ be the Gaussian ${\mathcal N}(a, \sigma^2)$ density
with some $\sigma^2>0$, and let $a_0<a_1$. For $f_1= \mathsf{ f}_{a_1}$
and $f_0 = \mathsf{ f}_{a_0}$, the log-likelihood ratio
\[
\log\frac{f_1}{f_0}(X) = X \frac{a_1-a_0}{\sigma^2}- \frac
{a_1^2 - a_0^2}{2 \sigma^2}
\]
is increasing in $X$. By Theorem~\ref{thm:st}, the
selector $\hat\eta$ on the class $\Theta_d^+(s, a_0,a_1)$ is a
vector with components
%
\begin{equation}
\label{eta01} \hat\eta_j = I \bigl( X_j \geq
t(a_0,a_1) \bigr),\qquad  j=1,\dots,d,
\end{equation}
where
\[
t(a_0,a_1) =\frac{a_1+a_0}2 + \frac{\sigma^2 \log(d/s-1)}{a_1-a_0}.
\]
Note that for $a_0=0$ it coincides with the selector in \eqref
{selector+} with $a=a_1$, which is minimax optimal on $\Theta_d^+(s,
a_1)$. Moreover, the minimax risk only depends on $a_0$ and $a_1$
through the difference $\delta=a_1-a_0$:
\[
\Psi= \Phi \biggl(-\frac{\delta}2 + \frac{\sigma^2 \log
(d/s-1)}{\delta} \biggr)+ \biggl(
\frac{d}s - 1 \biggr) \Phi \biggl(-\frac
{\delta}2 + \frac{\sigma^2 \log(d/s-1)}{\delta}
\biggr).
\]
\end{example}

\begin{example}\label{exmp2}
Let $\mathbb{P}_a$ be the Bernoulli distribution $B(a)$ with parameter
$a \in
(0,1)$, and $0< a_0 < a_1 <1$. Denoting by $\mathsf{ f}_a$ the density of
$\mathbb{P}_a$ with respect to the counting measure we have, for $f_1=
\mathsf{f}_{a_1}$ and $f_0 = \mathsf{ f}_{a_0}$,
\[
\log\frac{f_1}{f_0}(X) = X \log \biggl(\frac{a_1}{1-a_1} \frac
{1-a_0}{a_0}
\biggr) + \log\frac{1-a_1}{1-a_0}
\]
which is increasing in $X$ for $0< a_0 < a_1 <1$. The nearly minimax optimal
selector $\hat\eta$ on the class $\Theta_d^+(s, a_0,a_1)$ is a
vector with components
$\hat\eta_j$ in \eqref{eta01}
where the threshold $t(a_0,a_1)$ is given by
\[
t(a_0,a_1)= \frac{ \log( \frac{d}s-1) - \log\frac{1-a_1}{1-a_0}}{\log(\frac{a_1}{1-a_1} \frac{1-a_0}{a_0})} .
\]
Note that the nearly minimax selector $\hat\eta_j$ differs from the naive
selector $\hat\eta_j^{n}= X_j$. Indeed since $X_j\in\{0,1\}$ we have
$\hat\eta_j=1$ if either $X_j=1$ or $t(a_0,a_1)\le0$, and $\hat\eta
_j=0$ if either $X_j=0$ or $t(a_0,a_1)>1$.
The value $\Psi$ in the risk has the form
\begin{eqnarray*}
\Psi&=& \mathbb{P}_{a_1}\bigl(X_{1} < t(a_0,a_1)
\bigr) + \biggl(\frac{d}s - 1 \biggr) \mathbb{P} _{a_0}\bigl(X_{1}
\ge t(a_0,a_1)\bigr)
\\
&=& \begin{cases} d/s - 1, & t(a_0,a_1) \le0,
\\
1- a_1 + a_0 ( d/s - 1), & 0 < t(a_0,a_1)
<1,
\\
1, & t(a_0,a_1) \geq1. \end{cases}
\end{eqnarray*}
In the asymptotic regime when $d\to\infty$ and $s\to\infty$, the
minimax risk is of order $s\Psi$ and can converge to 0 only when
the parameters $d,s, a_0,a_1$ are kept such that $0 < t(a_0,a_1) <1$,
and in addition $(1-a_1) s\to0$, $a_0(d-s) \to0$.
Thus, the risk can converge to 0 only when the Bernoulli probabilities
$a_1$ and $a_0$ tend sufficiently fast to 1 and to 0, respectively.
\end{example}

\begin{example}\label{exmp3}
Let $\mathbb{P}_a$ be the Poisson distribution with parameter $a>0$,
and let
$a_1>a_0 >0$. Denoting by $\mathsf{ f}_a$ the density of $\mathbb{P}_a$ with
respect to the counting measure we have
\[
\log\frac{f_1}{f_0}(X) = X\log \biggl(\frac{a_1}{a_0} \biggr)
-a_1+a_0,
\]
which is increasing in $X$. The components of the nearly minimax optimal
selector $\hat\eta$ are given by \eqref{eta01} with
\[
t(a_0,a_1)=\frac{\log( d/s - 1) +a_1 - a_0}{\log(a_1/a_0)}.
\]
Note that $t(a_0,a_1)>0$ as soon as $d/s \geq2$ and $a_1>a_0 >0$. The value of
$\Psi$ in the risk has the form
$
\Psi= \mathbb{P}_{a_1}(X_{1} < t(a_0,a_1)) +(d/s - 1) \mathbb{P}_{a_0}(X_{1}
\ge t(a_0,a_1))$.
\end{example}

\subsection{Crowdsourcing with sparsity constraint} The problem of
crowdsourcing with two classes is a clustering problem that can be
formalized as follows; cf. \cite{GLZ}. Assume that $m$ workers provide
class assignments for $d$ items. The class assignment $X_{ij}$ of the
$i$th worker for the $j$th item is assumed to have a Bernoulli
distribution $B(a_{i0})$ if the $j$th item belongs to class 0, and a
Bernoulli distribution $B(a_{i1})$ if it belongs to class 1. Here,
$a_{i0},a_{i1}\in(0,1)$ and $a_{i0} \ne a_{i1}$ for $i=1,\dots,m$.
The observations $(X_{ij},i=1,\dots,m, j=1,\dots,d)$ are assumed to
be jointly independent. Thus, each vector $X_j=(X_{1j},\dots,X_{mj})$
is distributed according to $P_{0}$ or to $P_{1}$
where each of
these two measures is a product of Bernoulli measures, and $P_{0}\ne
P_{1}$. We assume that there are at most $s$ vectors $X_j$ with distribution
$P_{1}$, and the other vectors $X_j$ with distribution $P_{0}$. The
aim is to recover the binary vector of class labels $\eta=(\eta
_1,\dots,\eta_d)$ based on the observations ${\mathbf X}=(X_1,\dots,X_d)$. Here, $\eta_j\in\{0,1\} $ satisfies $\eta_j=k$ if the $j$th
item belongs to class $k\in\{0,1\}$. Thus, we are in the framework of
Theorem~\ref{thm:fixed} with a particular form of the log-likelihood ratio
%
\begin{equation}
\label{cor:crowdsourcing:eq1} \log\frac{f_1}{f_0}(X_{j}) = \sum
_{i=1}^m \biggl(X_{ij} \log \biggl(
\frac{a_{i1}}{1-a_{i1}} \frac{1-a_{i0}}{a_{i0}} \biggr) + \log\frac
{1-a_{i1}}{1-a_{i0}} \biggr),
\end{equation}
where $f_k$ is the density of $P_{k}$, $k\in\{0,1\}$. The following
corollary is an immediate consequence of Theorem~\ref{thm:fixed}.
%
\begin{corollary}\label{cor:crowdsourcing}
Let $s<d$, $a_{i0},a_{i1}\in(0,1)$\vspace*{1pt} and $a_{i0} \ne a_{i1}$ for
$i=1,\dots,m$. Then the selector $\hat\eta$ in (\ref{def:newsel})
with $\log\frac{f_1}{f_0}(X_{j})$ defined in \eqref
{cor:crowdsourcing:eq1} satisfies Theorem~\ref{thm:fixed}. 
\end{corollary}
%
For suitable combinations of parameters $d,s, a_{i0},a_{i1}$, the exact asymptotic
value of the minimax risk $\Psi$ can be further analyzed to obtain
asymptotics of interest. Gao et al. \cite{GLZ} have studied a setting
of crowdsourcing problem which is different from the one we consider
here. They did not assume sparsity $s$, and instead of the class
$\Theta_d(s, f_0,f_1)$ of at most $s$-sparse binary sequences, they considered
the class of all possible binary sequences $\{0,1\}^d$. For this class,
Gao et al. \cite{GLZ} 
analyzed specific asymptotics of the minimax risk $\inf_{\tilde
\eta} \sup_{\eta\in\{0,1\}^d} d^{-1} \mathbf{E} \llvert \tilde
\eta-
\eta \rrvert $ in large deviations perspective.

\section{Asymptotic analysis. Phase transitions}\label{sec:asymp}

In this section, we conduct the asymptotic analysis of the problem of
variable selection. The results are derived as corollaries of the
minimax bounds of Section~\ref{sec:minimax}. We will assume that $d\to
\infty$ and that parameters $a=a_d$ and $s=s_d$ depend on $d$.

The first two asymptotic properties we study here are \textit{exact
recovery} and \textit{almost full recovery}. We use this terminology
following \cite{GJW2012,JJ2012} but we define these properties in a
different way, as asymptotic minimax properties for classes of vectors~$\theta$.
The papers \cite{GJW2012,JJ2012} considered a Bayesian
setup with random $\theta$ and studied a linear regression model with
i.i.d. Gaussian regressors rather than the sequence model \eqref{vectormodel}.

The study of \textit{exact recovery} and \textit{almost full recovery} will
be done here only for the classes $\Theta_d(s_d,a_d)$. The
corresponding results for the classes $\Theta_d^+ (s_d,a_d)$ or
$\Theta_d^- (s_d,a_d)$ are completely analogous.
We do not state them here for the sake of brevity.

\begin{definition} Let $(\Theta_d(s_d,a_d))_{d\ge1}$ be a sequence of
classes of sparse vectors:
\begin{itemize}
\item We say that
\emph{exact recovery is possible} for $(\Theta_d(s_d,a_d))_{d\ge
1}$ if there exists a selector $\hat\eta$ such that
%
\begin{gather}
\label{er1} \lim_{d\to\infty} \sup_{\theta\in\Theta_d(s_d,a_d)}\mathbf
{E}_{\theta} \llvert \hat{\eta}-\eta \rrvert =0.
\end{gather}
In this case, we say that $\hat\eta$ achieves exact recovery.

\item We say that \emph{almost full recovery is possible} for
$(\Theta_d(s_d,a_d))_{d\ge1}$ if there exists a selector $\hat\eta$
such that
%
\begin{gather}
\label{af1} \lim_{d\to\infty} \sup_{\theta\in\Theta_d(s_d,a_d)}
\frac{1}{s_d} \mathbf{E}_{\theta} \llvert \hat{\eta}-\eta \rrvert =0.
\end{gather}
In this case, we say that $\hat\eta$ achieves almost full recovery.
\end{itemize}
\end{definition}

It is of interest to characterize the sequences $(s_d,a_d)_{d\ge1}$,
for which exact recovery and almost full recovery are possible.
To describe the impossibility of exact or almost full recovery, we need
the following definition. 

\begin{definition} Let $(\Theta_d(s_d,a_d))_{d\ge1}$ be a sequence of
classes of sparse vectors:
\begin{itemize}
\item We say that
\emph{exact recovery is impossible} for $(\Theta_d(s_d,a_d))_{d\ge
1}$ if
%
\begin{gather}
\label{er2} \liminf_{d \to\infty} \inf_{\tilde{\eta}}\sup
_{\theta\in
\Theta_d(s_d,a_d)}\mathbf{E}_{\theta} \llvert \tilde{\eta}-\eta
\rrvert >0,
\end{gather}
\item We say that \emph{almost full recovery is impossible} for
$(\Theta_d(s_d,a_d))_{d\ge1}$ if
%
\begin{gather}
\label{af2} \liminf_{d \to\infty} \inf_{\tilde{\eta}}\sup
_{\theta\in
\Theta_d(s_d,a_d)} \frac{1}{s_d} \mathbf{E}_{\theta} \llvert
\tilde {\eta}-\eta \rrvert >0,
\end{gather}
where $ \inf_{\tilde{\eta}}$ denotes the infimum over all selectors.
\end{itemize}
\end{definition}

The following general characterization theorem is a straightforward
corollary of the results of Section~\ref{sec:minimax}.

\begin{theorem}\label{t3a}
\textup{(i)}
Almost full recovery is possible for $(\Theta_d(s_d,a_d))_{d\ge1}$ if
and only if $s_d \to \infty$ and
%
\begin{equation}
\label{eq1:t3a} \Psi_+(d, s_d,a_d)\to0 \qquad\text{as } d
\to\infty.
\end{equation}
In this case, the selector $\hat\eta$ defined in \eqref{selector}
with threshold
\eqref{threshold} achieves almost full recovery.

\textup{(ii)}
Exact recovery is possible for $(\Theta_d(s_d,a_d))_{d\ge1}$ if and
only if
%
\begin{equation}
\label{eq2:t3a} s_d\Psi_+(d, s_d,a_d)\to0
\qquad\text{as } d\to\infty.
\end{equation}
In this case, the selector $\hat\eta$ defined in \eqref{selector}
with threshold
\eqref{threshold} achieves exact recovery.
\end{theorem}

Although this theorem gives a complete solution to the problem,
conditions \eqref{eq1:t3a} and \eqref{eq2:t3a} are not quite
explicit. Intuitively, we would like to get a ``phase transition''
values $a_d^*$ such that exact (or almost full) recovery is possible
for $a_d$ greater than $a_d^*$ and is impossible for $a_d$ smaller than
$a_d^*$. Our aim now is to find such ``phase transition'' values. We
first do it in the almost full recovery framework.

The following bounds for the tails of Gaussian distribution will be useful:
%
\begin{equation}
\label{AS} \sqrt{\frac{2}\pi} \frac{e^{-y^2/2}}{y+\sqrt{y^2+4}} <
\frac
{1}{\sqrt{2\pi}}
\int_y^\infty e^{-u^2/2} \,du\leq\sqrt{
\frac{2}\pi} \frac{e^{-y^2/2}}{y+\sqrt{y^2+8/\pi}},
\end{equation}
for all $y\geq0$.
These bounds are an immediate consequence of formula 7.1.13 in \cite{AbrStegun} with $x= y/\sqrt{2}$.

Furthermore, we will need some nonasymptotic bounds for the expected
Hamming loss that will play a key role in the subsequent asymptotic
analysis. They are given in the next theorem.

\begin{theorem}\label{t4}
Assume that $s < d/2$.
\begin{itemize}
\item[(i)]
If
%
\begin{equation}
\label{ass:e} a^2 \ge\sigma^2 \bigl(2 \log\bigl((d-s)/s
\bigr) + W \bigr)\qquad \mbox{for some } W>0,
\end{equation}
then the selector $\hat\eta$ defined in \eqref{selector} with threshold
\eqref{threshold} satisfies
%
\begin{equation}
\label{eq:t4:1} \sup_{\theta\in\Theta_d(s,a)} \mathbf{E}_\theta \llvert
\hat \eta- \eta \rrvert \leq(2 + \sqrt{2 \pi}) s \Phi(-\Delta),
\end{equation}
where $\Delta$ is defined by
%
\begin{equation}
\label{DeltaERec} \Delta= \frac{W}{2 \sqrt{2 \log((d-s)/s) +W}} .
\end{equation}
\item[(ii)] If $a>0$ is such that
%
\begin{equation}
\label{lowE} a^2 \le\sigma^2 \bigl(2 \log\bigl((d-s)/s
\bigr) + W \bigr)\qquad \mbox{for some } W>0,
\end{equation}
then, for any $s'$ in $(0,s]$ we have
%
\begin{equation}
\label{eq:t4:2} \inf_{\tilde\eta} \sup_{\theta\in\Theta_d(s,a)} \mathbf
{E}_\theta \llvert \tilde\eta- \eta \rrvert \geq s' \Phi (-\Delta)
-4 s' \exp\left( -\frac{(s-s')^2}{2s}\right),
\end{equation}
where the infimum is taken over all selectors $\tilde\eta$ and
$\Delta>0$ is defined in \eqref{DeltaERec}.
\end{itemize}
\end{theorem}
The proof is given in the Appendix.



The next theorem is an easy consequence of Theorem~\ref{t4}. It
describes a ``phase transition'' for $a_d$ in the problem of almost full
recovery.

\begin{theorem}\label{thm:asymp}
Assume that $\limsup_{d\to\infty} s_d/d<1/2$:
\begin{itemize}
\item[(i)] If, for all $d$ large enough,
\[
a^2_d \ge\sigma^2 \bigl(2\log
\bigl((d-s_d)/s_d\bigr) + A_d \sqrt{2\log
\bigl((d-s_d)/s_d\bigr)} \bigr)
\]
for an arbitrary sequence $A_d\to\infty$, as $d\to\infty$,
then the selector $\hat\eta$ defined by \eqref{selector} and \eqref
{threshold} achieves almost full recovery:
\[
\lim_{d \to\infty}\sup_{\theta\in\Theta_d(s_d,a_d)}\frac{1}{s_d}
\mathbf{E}_\theta \llvert \hat\eta- \eta \rrvert = 0.
\]
\item[(ii)] Moreover, if there exists $A>0$ such that for all $s$ and $d$
large enough the reverse inequality holds:
\begin{equation}
\label{eq2:thm:asympE} a^2_d \le\sigma^2 \bigl(2\log
\bigl((d-s_d)/s_d\bigr) + A \sqrt{2\log
\bigl((d-s_d)/s_d\bigr)} \bigr)
\end{equation}
then almost full recovery is impossible:
\begin{equation*}
\liminf_{d \to\infty} \inf_{\tilde{\eta}}\sup
_{\theta\in
\Theta_d(s_d,a_d)} \frac{1}{s_d} \mathbf{E}_{\theta} \llvert
\tilde {\eta}-\eta \rrvert 
>0.
\end{equation*}
Here, $\inf_{\tilde{\eta}}$ is the infimum over all selectors
$\tilde\eta$.
\end{itemize}
\end{theorem}
The proof is given in the Appendix.

Inspection of the proof shows that the lower bound in Theorem~\ref{thm:asymp} holds true for an arbitrary $s_d\geq 5$ (possibly fixed), if (\ref{eq2:thm:asympE}) is satisfied for some $A$ in (0,1).

Under the sparsity assumption that
%
\begin{equation}
\label{condsp}s_d \to \infty, \, d/s_d \to\infty\qquad\mbox{as } d \to\infty,
\end{equation}
Theorem~\ref{thm:asymp} shows that the ``phase transition'' for almost
full recovery occurs at the value $a_d=a_d^*$, where
%
\begin{equation}
\label{phase} a_d^* = \sigma\sqrt{2\log\bigl((d-s_d)/s_d
\bigr)} \bigl(1+o(1) \bigr).
\end{equation}
Furthermore, Theorem~\ref{thm:asymp} details the behavior of the
$o(1)$ term here.

We now state a corollary of Theorem~\ref{thm:asymp} under simplified
assumptions.

\begin{corollary}\label{cor1}
Assume that \eqref{condsp} holds
and set
\[
a_d=\sigma\sqrt{2(1+\delta) \log(d/s_d)} \qquad\mbox{for
some } \delta>0.
\]
Then the selector $\hat\eta$ defined by \eqref{selector} with
threshold $t=\sigma\sqrt{2(1+\varepsilon(\delta)) \log(d/s_d)}$
where $\varepsilon(\delta)>0$ depends only on $\delta$, achieves
almost full recovery.
\end{corollary}
In the particular case of $s_d = d^{1-\beta}(1+o(1))$ for some $\beta
\in(0,1)$, condition \eqref{condsp} is satisfied. Then $\log(d /
s_d)=\beta(1+o(1)) \log d$ and it follows from Corollary~\ref{cor1}
that for $a_d=\sigma\sqrt{2\beta(1+\delta) \log d}$ the selector
with components $\hat\eta_j = I ( \llvert  X_j \rrvert >
\sigma\sqrt{2\beta(1+\varepsilon) \log d} )$ achieves almost
full recovery. This is in agreement with the findings of \cite
{GJW2012,JJ2012} where an analogous particular case of $s_d$ was
considered for a different model and the Bayesian definition of almost
full recovery.

We now turn to the problem of exact recovery. First, notice that if
\[
\limsup_{d\to\infty}s_d<\infty
\]
the properties of exact recovery and almost full recovery are
equivalent. Therefore, it suffices to consider exact recovery only when
$s_d\to\infty$ as $d\to\infty$. Under this assumption,
the ``phase transition'' for $a_d$ in the problem of exact recovery is
described in the next theorem.

\begin{theorem}\label{thm:asympE}
Assume that $\lim_{d\to\infty} s_d =\infty$ and $\limsup_{d\to
\infty} s_d/d<1/2$.
\begin{itemize}
\item[(i)]
If
\[
a^2_d \ge\sigma^2 \bigl(2 \log
\bigl((d-s_d)/s_d\bigr) + W_d \bigr)
\]
for all $d$ large enough, where the sequence $W_d$ is such that
%
\begin{equation}
\label{eq1:thm:asympE} \liminf_{d\to\infty} \frac{W_d}{4  (\log(s_d) + \sqrt{\log
(s_d) \log(d-s_d)} )} \ge1,
\end{equation}
then the selector $\hat\eta$ defined by \eqref{selector} and \eqref
{threshold} achieves exact recovery:
%
\begin{equation}
\label{eq2a:thm:asymp} \lim_{d \to\infty}\sup_{\theta\in\Theta_d(s_d,a_d)} \mathbf
{E}_\theta \llvert \hat\eta- \eta \rrvert = 0.
\end{equation}
\item[(ii)] If the complementary condition holds,
\[
a^2_d \le\sigma^2 \bigl(2 \log
\bigl((d-s_d)/s_d\bigr) + W_d \bigr)
\]
for all $d$ large enough, where the sequence $W_d$ is such that
%
\begin{equation}
\label{eq2:thm:asymp} \limsup_{d\to\infty} \frac{W_d}{4  (\log(s_d) + \sqrt{\log
(s_d) \log(d-s_d)} )} < 1,
\end{equation}
then exact recovery is impossible, and moreover we have
\begin{equation*}
\lim_{d\to\infty} \inf_{\tilde{\eta}}\sup
_{\theta\in
\Theta_d(s_d,a_d)} \mathbf{E}_{\theta} \llvert \tilde{\eta}-\eta
\rrvert = \infty.
\end{equation*}
Here, $\inf_{\tilde{\eta}}$ is the infimum over all selectors
$\tilde\eta$.
\end{itemize}
\end{theorem}
The proof is given in the Appendix.

Some remarks are in order here. First of all, Theorem~\ref{thm:asympE}
shows that
the ``phase transition'' for exact recovery occurs at $W_d=4  (\log
(s_d) + \sqrt{\log(s_d) \log(d-s_d)} )$, which corresponds to
the critical value $a_d=a_d^*$ of the form
%
\begin{equation}
\label{phase1} a_d^* = \sigma \bigl(\sqrt{2\log(d-s_d)}
+ \sqrt{2\log s_d} \bigr).
\end{equation}
This value is greater than the critical value $a_d^*$ for almost full
recovery [cf. \eqref{phase}], which is intuitively quite clear.
The optimal threshold \eqref{threshold} corresponding to \eqref
{phase1} has a simple form:
\[
t_d^* = \frac{a_d^*}2+\frac{\sigma^2}{a_d^*}\log \biggl(
\frac
{d}{s_d} -1 \biggr)= \sigma\sqrt{2\log(d-s_d)}.
\]
For example, if $s_d = d^{1-\beta}(1+o(1))$ for some $\beta\in
(0,1)$, then
$a_d^* \sim\sigma(1+ \sqrt{1-\beta}) \sqrt{2 \log d}$. In this
particular case, Theorem~\ref{thm:asympE} implies that if $a_d=\sigma
(1+ \sqrt{1-\beta})\sqrt{2(1+\delta) \log d}$ for some $\delta>0$,
then exact recovery is possible and the selector with threshold $t =
\sigma\sqrt{2 (1+\varepsilon) \log d}$ for some $\varepsilon>0$
achieves exact recovery. This is in agreement with the results of \cite
{GJW2012,JJ2012} where an analogous particular case of $s_d$ was
considered for a different model and the Bayesian definition of exact
recovery. For our model, even a sharper result is true; namely, a
simple universal threshold $t = \sigma\sqrt{2 \log d}$ guarantees
exact recovery adaptively in the parameters $a$ and $s$. Intuitively,
this is suggested by the form of $t_d^*$. The precise statement is
given in Theorem~\ref{thm:adaptE} below.

Finally, we state an asymptotic corollary of Theorem~\ref
{cor:recovery_pattern} showing that the selector $\hat\eta$
considered above is sharp
in the asymptotically minimax sense with respect to the risk defined as
the probability of wrong recovery.

\begin{theorem}\label{th:sharp:asymp}
Assume that exact recovery is possible for the classes $(\Theta
_d(s_d,a_d))_{d\ge1}$ and $(\Theta_d^+(s_d,a_d))_{d\ge1}$, that is,
condition \eqref{eq2:t3a} holds. Then, for the selectors $\hat\eta$
and $\hat\eta^+$ defined by \eqref{selector}, \eqref{selector+} and
\eqref{threshold}, and for the selector $\bar\eta$ defined by
\eqref{selector2} and \eqref{threshold2}, we have
\begin{equation*}
\begin{aligned}
\lim_{d\to\infty} \sup_{\theta\in\Theta_d^+(s_d,a_d)} \frac
{\mathbf{P}
_\theta(S_{\hat\eta^+} \ne S(\theta))}{s_d\Psi_+(d,s_d,a_d)}&=
\lim_{d\to\infty} \inf_{\tilde\eta\in\mathcal T} \sup
_{\theta
\in\Theta_d^+(s_d,a_d)} \frac{\mathbf{P}_\theta(S_{\tilde\eta} \ne
S(\theta))}{s_d\Psi_+(d,s_d,a_d)} = 1,
\\
\lim_{d\to\infty} \sup_{\theta\in\Theta_d(s_d,a_d)} \frac
{\mathbf{P}
_\theta(S_{\bar\eta} \ne S(\theta))}{s_d \bar\Psi
(d,s_d,a_d)}&=
\lim_{d\to\infty} \inf_{\tilde\eta\in\mathcal
T} \sup
_{\theta\in\Theta_d(s_d,a_d)} \frac{\mathbf{P}_\theta
(S_{\tilde\eta} \ne S(\theta))}{s_d \bar\Psi(d,s_d,a_d)} = 1
\end{aligned}
\end{equation*}
and
\[
\limsup_{d\to\infty} \sup_{\theta\in\Theta_d(s_d,a_d)}
\frac
{\mathbf{P}_\theta(S_{\hat\eta} \ne S(\theta))}{s_d\Psi
_+(d,s_d,a_d)}\le2.
\]
\end{theorem}

Note that the threshold \eqref{threshold} depends on the parameters
$s$ and $a$, so that the selectors considered in all the results above
are not adaptive. In the next section, we propose adaptive selectors
that achieve almost full recovery and exact recovery without the
knowledge of $s$ and $a$.

\begin{remark}\label{exhsearch}
Another procedure of variable selection is the exhaustive search
estimator of the support $S(\theta)$ defined as
\[
\tilde S = \mathop{\operatorname{argmax}}_{C \subseteq\{ 1,\dots,d \}:
 \llvert  C \rrvert =s} \sum
_{j \in C} X_j.
\]
This estimator was studied by Butucea \textit{et al.} \cite{BIS}. The
selection procedure can be equivalently stated as choosing the indices
$j$ corresponding to $s$ largest order statistics of the sample
$(X_1,\dots,X_d)$. In \cite{BIS}, Theorem~2.5, it was shown that, on
the class $\Theta_d^+(s_d,a_d)$, the probability of wrong recovery
$\mathbf{P}_\theta(\tilde S \neq S(\theta))$ tends to 0 as $d\to
\infty$
under a stronger condition on $(s_d,a_d)$ than \eqref{eq2:t3a}. The
rate of this convergence was not analyzed there. If we denote by $\eta
_{\tilde S} $ the selector with components $I(j \in\tilde S)$ for $j$
from 1 to $d$, it can be proved that $\mathbf{E} \llvert \eta
_{\tilde S} -
\eta \rrvert \leq2 \mathbf{E} \llvert \hat\eta^+ - \eta
 \rrvert $, and thus the risk of $\eta_{\tilde S}$ is within at least a factor 2
of the minimax risk over the class $\Theta^+_d(s, a)$.
\end{remark}

\section{Adaptive selectors}\label{sec:adapt}

In this section, we consider the asymptotic setup as in Section~\ref{sec:asymp} and construct the selectors
that provide almost full and exact recovery adaptively, that is,
without the knowledge of $a$ and $s$.

As discussed in Section~\ref{sec:asymp}, the issue of adaptation for
exact recovery is almost trivial. Indeed, the expressions for minimal
value $a_d^*$, for which exact recovery is possible [cf. \eqref
{phase1}], and for the corresponding optimal threshold $t_d^*$ suggest
that taking a selector with the universal threshold $t = \sigma\sqrt
{2 \log d}$ is enough to achieve exact recovery simultaneously for all
values $(a_d,s_d)$, for which the exact recovery is possible. This
point is formalized in the next theorem.

\begin{theorem}\label{thm:adaptE}
Assume that $s_d\to\infty$ as $d\to\infty$ and that $\limsup_{d\to
\infty} s_d/d<1/2$. Let the sequence $(a_d)_{d\ge1}$ be above the
phase transition level for exact recovery, that is, $a_d\ge a_d^*$ for
all $d$, where $a_d^*$ is defined in \eqref{phase1}. Then the selector
$\hat\eta$ defined by \eqref{selector} with threshold $t = \sigma
\sqrt{2 \log d}$ achieves exact recovery.
\end{theorem}

The proof of this theorem is given in the Appendix.


We now turn to the problem of adaption for almost full recovery.
Ideally, we would like to construct a selector that achieves almost
full recovery for all sequences $(s_d, a_d)_{d\ge1}$ for which almost
full recovery is possible. We have seen in Section~\ref{sec:asymp}
that this
includes a much broader range of values than in
case of exact recovery.
Thus, using the adaptive selector of Theorem~\ref{thm:adaptE} for
almost full recovery
does not give a satisfactory result, and we have to take a different approach.\looseness=1

Following Section~\ref{sec:asymp}, we will use the notation
\[
a_0(s,A)\triangleq\sigma \bigl(2\log\bigl((d-s)/s\bigr) + A\sqrt{
\log \bigl((d-s)/s\bigr)} \bigr)^{1/2}.
\]
As shown in Section~\ref{sec:asymp}, it makes sense to consider the
classes $\Theta_d(s,a)$ only when $a\ge a_0(s,A)$ with some $A>0$,
since for other values of $a$ almost full recovery is impossible. Only
such classes will be studied below.

In the asymptotic setup of Section~\ref{sec:asymp}, we have used the
assumption that $d/s_d\to\infty$ (the sparsity assumption), which is
now transformed into the condition
%
\begin{equation}
\label{smax}
\begin{aligned}
&s_d\in{\mathcal S}_d\triangleq\bigl\{1,2,
\dots,s^*_d\bigr\}\\
&\qquad  \mbox{where $s^*_d$ is an integer
such that } \frac{d}{s^*_d}\to\infty\mbox{ as } d\to\infty.
\end{aligned}
\end{equation}
Assuming $s_d$ to be known, we have shown in Section~\ref{sec:asymp}
that almost full recovery is achievable for all $a\ge a_0(s_d,A_d)$,
where $A_d$ tends to infinity as $d\to\infty$. The rate of growth of
$A_d$ was allowed to be arbitrarily slow there; cf. Theorem~\ref
{thm:asymp}. However, for adaptive estimation considered in this
section we will need the following mild assumption on the growth of $A_d$:
%
\begin{equation}
\label{a_d} A_d\ge c_0 \biggl(\log\log \biggl(
\frac{d}{s^*_d} -1 \biggr) \biggr)^{1/2}, 
\end{equation}
where $c_0>0$ is an absolute constant. In what follows, we will assume
that $s^*_d\le d/4$, so that the right-hand side of \eqref{a_d} is
well defined.

Consider a grid of points $\{g_1, \dots, g_M\}$ on ${\mathcal S}_d$,
where $g_j=2^{j-1}$ and $M$ is the maximal integer such that $g_M\le
s^*_d$. For each $g_m$, $m=1,\dots,M$, we define a selector:
\[
\hat\eta(g_m) = \bigl(\hat\eta_j (g_m)
\bigr)_{j=1,\dots,d} \triangleq \bigl( I \bigl( \llvert X_j \rrvert
\geq w(g_m)\bigr) \bigr)_{j=1,\dots,d},
\]
where
\[
w(s)=\sigma\sqrt{2 \log \biggl(\frac{d}s - 1 \biggr)}.
\]
Note that $w(s)$ is monotonically decreasing.
We now choose the ``best'' index $m$, for which $g_m$ is near the true
(but unknown) value of $s$,
by the following data-driven procedure:
\begin{equation}
\label{def:hat:m} 
\begin{aligned}
\hat{m} &= \min \Biggl\{ m\in\{2,\dots,M\}:
\\
&\quad  \sum_{j=1}^d I \bigl(w(g_k)
\le \llvert X_j \rrvert < w(g_{k-1}) \bigr) \leq\tau
g_k \mbox{ for all } k\ge m +1\Biggr\},
\end{aligned}
\end{equation}
where
\[
\tau= \bigl( \log \bigl({d}/{s^*_d} -1 \bigr) \bigr)^{-\frac{1}{7}},
\]
and we set $\hat{m} =M$ if the set in \eqref{def:hat:m} is empty.
Finally, we define an adaptive selector as
\[
\hat\eta^{\mathrm{ad}}= \hat\eta(g_{\hat{m}}).
\]
This adaptive procedure is quite natural in the sense that it can be
related to the Lepski method or to wavelet thresholding that are widely
used for adaptive estimation. Indeed, as in wavelet methods, we
consider dyadic blocks determined by the grid points $g_j$. The value
$\sum_{j=1}^d I  (w(g_k)\le \llvert  X_j \rrvert <
w(g_{k-1}) )$ is the number of observations within the $k$th block.
If this number is too small (below a suitably chosen threshold), we
decide that the block corresponds to pure noise and it is rejected; in
other words, this $k$ is not considered as a good candidate for
$\hat{m}$. This argument is analogous to wavelet thersholding. We
start from the largest $k$ [equivalently, smallest $w(g_k)$] and
perform this procedure until we find the first block, which is not
rejected. The corresponding value $k$ determines our choice of
$\hat{m}$ as defined in \eqref{def:hat:m}.

\begin{theorem}\label{th:adapt}
Let $ {c_0\ge16}$. Then the selector $\hat\eta^{\mathrm{ad}}$
adaptively achieves almost
full recovery in the following sense:
%
\begin{equation}
\label{eq:th:adapt} \lim_{d\to\infty} \sup_{\theta\in\Theta_d(s_d,a_d)}
\frac{1}{s_d} \mathbf{E}_\theta \bigl\llvert \hat
\eta^{\mathrm{ad}}- \eta \bigr\rrvert = 0
\end{equation}
for all sequences $(s_d,a_d)_{d\ge1}$ such that \eqref{smax} holds
and $a_d\ge a_0(s_d,A_d)$, where $A_d$ satisfies \eqref{a_d}.
\end{theorem}

\begin{remark}
Another family of variable selection methods originates from the theory
of multiple testing \cite{ab,abdj}. These are, for example, the Benjamini--Hochberg,
Benjamini--Yekutieli or SLOPE procedures. We refer to \cite{Bogdan}
for a recent overview and comparison of these techniques. They have the
same structure as the exhaustive search procedure in that they keep
only the largest order statistics. The difference is that the value $s$
(which is usually not known in practice) is replaced by an estimator
$\hat s$ obtained from comparing the $i$th order statistic of $( \llvert  X_1 \rrvert ,\dots, \llvert  X_d \rrvert )$ with a
suitable normal quantile depending on $i$. The analysis of these
methods in the literature is focused on the evaluation of false
discovery rate (FDR). Asymptotic power calculations for the
Benjamini--Hochberg procedure are given in \cite{Arias-Castro}. To the
best of our knowledge, the behavior of the risk $\mathbf{P}_\theta
(\tilde S
\neq S(\theta))$ and of the Hamming risk, even in a simple
consistency perspective, was not studied.\looseness=1
\end{remark}

\begin{remark}In this paper, the variance $\sigma$ was supposed to be
known. Extension to the case of unknown $\sigma$ can be treated as
described, for example, in \cite{collier_etal2016}. Namely, we replace
$\sigma$ in the definition of the threshold $w(s)$ by a statistic
$\hat\sigma$ defined in \cite{collier_etal2016}, Section~3. As shown
in \cite{collier_etal2016}, Proposition~1, this statistic is such that
$\sigma\le\hat\sigma\le C'\sigma$ with high probability provided
that $s\le d/2$, and $d\ge d_0$ for some absolute constants $C'>1,
d_0 \geq1$. Then, replacing $\sigma$ by $\hat\sigma$ in the
expression for $w(s)$, one can show that Theorem~\ref{th:adapt}
remains valid with this choice of $w(s)$ independent of $\sigma$, up
to a change in numerical constants in the definition of the adaptive
procedure. With this modification, we obtain a procedure which is
completely data-driven and enjoys the property of almost full recovery
under the mild conditions given in Theorem~\ref{th:adapt}. The same
modification can be done in Theorem~\ref{thm:adaptE}. Namely, under
the assumptions of Theorem~\ref{thm:adaptE} and $a_d\ge c' a_d^*$,
where $c'\ge1$ is a numerical constant, the selector $\hat\eta$
defined by \eqref{selector} with threshold $t = \hat\sigma\sqrt{2
\log d}$ achieves exact recovery when $\sigma$ is unknown.\looseness=1
\end{remark}
%
\begin{remark}
In this section, the problem of adaptive variable selection was
considered only for the classes $\Theta_d(s_d,a_d)$. The corresponding
results for classes $\Theta_d^+ (s_d,a_d)$ and $\Theta_d^- (s_d,a_d)$
are completely analogous.
We do not state them here for the sake of brevity.
\end{remark}

%

\begin{appendix}
\section*{Appendix}\label{app:Appendix}

\begin{proof}[Proof of Theorem~\ref{nonasymptotic}]
We have, for any $t>0$,
\begin{eqnarray*}
\llvert \hat\eta- \eta \rrvert & = & \sum_{j: \eta_j=0}
\hat\eta_j +\sum_{j: \eta_j=1} (1- \hat
\eta_j )
\\
&=& \sum_{j: \eta_j=0} I\bigl( \llvert \sigma
\xi_j \rrvert \geq t \bigr) +\sum_{j: \eta_j=1}
I \bigl( \llvert \sigma\xi_j + \theta_j \rrvert < t
\bigr).
\end{eqnarray*}
Now, for any $\theta\in\Theta_d(s,a)$ and any $t>0$,
\begin{eqnarray*}
 \mathbf{E} \bigl(I \bigl( \llvert \sigma\xi_j +
\theta_j \rrvert < t \bigr) \bigr) &\leq&\mathbf{P}\bigl( \llvert
\theta_j \rrvert - \llvert \sigma \xi _j \rrvert < t
\bigr) \leq\mathbf{P}\bigl( \llvert \xi \rrvert > (a - t)/\sigma\bigr)
\\
&=& \mathbf{P}\bigl( \llvert \xi \rrvert > (a - t)_+ /\sigma\bigr),
\end{eqnarray*}
where $\xi$ denotes a standard Gaussian random variable.
Thus, for any $\theta\in\Theta_d(s,a)$,
%
\begin{equation}
\label{term1} \frac{1}s \mathbf{E}_\theta \llvert \hat\eta-
\eta \rrvert \leq \frac{d-|S|}s  \mathbf{P}\biggl( \llvert
\xi \rrvert \geq\frac{t} \sigma\biggr) + \frac{|S|}{s} \mathbf{P} \biggl( \llvert \xi
\rrvert >\frac{(a- t)_+}\sigma \biggr) \leq 2\Psi(d,s,a).
\end{equation}
%
Indeed, for $t$ defined in~(\ref{threshold}), $t \geq (a-t)_+$ given that $s\leq d/2$. Here and in the sequel, $|S|$ denotes the cardinality of $S= S(\theta)$.
\end{proof}


\begin{proof}[Proof of Theorem~\ref{t2}]
Arguing as in the proof of
Theorem~\ref{nonasymptotic}, we obtain
\begin{eqnarray*}
\bigl\llvert \hat\eta^+ - \eta \bigr\rrvert &=& \sum
_{j: \eta_j=0} I(\xi_j\geq t ) +\sum
_{j: \eta_j=1} I (\sigma\xi_j + \theta_j < t
),
\end{eqnarray*}
and
$\mathbf{E} (I  ( \sigma\xi_j + \theta_j < t  )
 )
\le\mathbf{P}(\xi< (t -a)/\sigma)$.
Thus, for any $\theta\in\Theta_d^+(s,a)$,
\begin{equation*}
\label{term1_1} \frac{1}s \mathbf{E}_\theta \bigl\llvert \hat
\eta^+ - \eta \bigr\rrvert \leq \frac{d-|S|}s  \mathbf{P}(\xi
\geq t/\sigma) + \frac{|S|}{s} \mathbf{P}\bigl(\xi < (t -a)/\sigma\bigr) \leq \Psi_+(d,s,a),
\end{equation*}
by the monotonicity of $\Phi$ and the condition $s\leq d/2$.
\end{proof}


\begin{proof}[Proof of Theorem~\ref{thm:lb}]
%
%
We prove here the first inequality of Theorem~\ref{thm:lb}. Since $\tilde \eta_j$ depends only on $X_j$,
%
\begin{equation}
\label{eq:lb1}
 \mathbf{E}_\theta \llvert \tilde\eta- \eta \rrvert =
 \sum
_{j=1}^d \mathbf{E}_{j,\theta_j} \llvert
\tilde\eta_j - \eta_j \rrvert,
\end{equation}
where $\mathbf{E}_{j,\theta_j}$ is the expectation with respect to the distribution of $X_j$.

Let $\Theta'$ be the set of all $\theta$ in $\Theta_d^+(s,a)$ such
that $s$ components $\theta_j$ of $\theta$ are equal to $a$ and the
remaining $d-s$ components are 0. Denote by $ \llvert \Theta' \rrvert ={d\choose s}$ the cardinality of $\Theta'$. Then, for any $\tilde \eta\in{\mathcal T}$ we have
%
\begin{equation}
\label{eq:lb2}
\begin{aligned}
& \sup_{\theta\in\Theta_d^+(s,a)}\frac{1}s \mathbf{E}_\theta \llvert \tilde\eta- \eta \rrvert\\
&\qquad  \geq\frac{1}{s \llvert \Theta' \rrvert } \sum
_{\theta\in
\Theta'} \sum_{j=1}^d
\mathbf{E}_{j,\theta_j} \llvert \tilde \eta _j - \eta_j
\rrvert
\\
&\qquad =  \frac{1}{s \llvert \Theta' \rrvert } \sum_{j=1}^d
\biggl(\sum_{\theta\in\Theta':\theta_j=0} {\mathbf{E}}_{j,0} (\tilde
\eta _j) + \sum_{\theta\in\Theta':\theta_j=a} {
\mathbf{E}}_{j,a} (1-\tilde\eta_j ) \biggr)
\\
&\qquad =  \frac{1}{s} \sum_{j=1}^d
\biggl( \biggl(1-\frac{s}d \biggr) {\mathbf{E}}_{j,0} (\tilde
\eta_j) + \frac{s}{d} {\mathbf{E}}_{j,a} (1-\tilde
\eta_j ) \biggr)
\\
&\qquad  \geq \frac{d}{s} \inf_{T \in[0,1]} \biggl( \biggl(1-
\frac{s}d \biggr) {\mathbb E}_0 (T)+ \frac{s}{d} {
\mathbb E}_a (1-T) \biggr),
\end{aligned}
\end{equation}
where we have used that $ \llvert \{\theta\in\Theta':\theta_j=a\}
 \rrvert ={d-1\choose s-1}=s \llvert \Theta' \rrvert /d$.
In the last line of display \eqref{eq:lb2}, ${\mathbb E}_u$ is
understood as the expectation with respect to the distribution of
$X=u+\sigma\xi$, where $\xi\sim{\mathcal N}(0,1)$ and $\inf_{T \in
[0,1]}$ denotes the infimum over all $[0,1]$-valued statistics $T(X)$. Set
\[
L^*=\inf_{T \in[0,1]} \biggl( \biggl(1-\frac{s}d \biggr) {
\mathbb E}_0 (T)+ \frac{s}{d} {\mathbb E}_a (1-T)
\biggr)
\]
By the Bayesian version of the Neyman--Pearson lemma, the infimum here
is attained for $T=T^*$ given by
\[
T^*( X) = I \biggl( \frac{(s/d) \varphi_{\sigma} (X-a)}{(1-s/d)
\varphi_{\sigma}(X)} >1 \biggr)
\]
where $ \varphi_{\sigma}(\cdot)$ is the density of an ${\mathcal
N}(0,\sigma^2)$ distribution. Thus,
\begin{eqnarray*}
L^* &=& \biggl(1-\frac{s}d \biggr) \mathbf{P} \biggl(
\frac{\varphi_{\sigma}
({\sigma}\xi-a)}{ \varphi_{\sigma}({\sigma}\xi)} > \frac{d}s - 1 \biggr) + \frac{s}d \mathbf{P}
\biggl( \frac{ \varphi_{\sigma}
({\sigma}\xi)}{ \varphi_{\sigma}({\sigma}\xi+a)} \leq\frac{d}s - 1 \biggr).
\end{eqnarray*}
Combining this with \eqref{eq:lb1} and \eqref{eq:lb2}, we get
\begin{align*}
&\inf_{\tilde\eta} \sup_{\theta\in\Theta_d^+(s,a)}
\frac{1}s \mathbf{E}_\theta \llvert \tilde\eta- \eta \rrvert
\nonumber
\\
&\qquad  \geq \biggl( \frac{d}s - 1 \biggr) \mathbf{P} \biggl( \exp \biggl(
\frac
{a\xi}{\sigma} - \frac{a^2}{2\sigma^2} \biggr) > \frac{d}s - 1 \biggr)\\
&\qquad \quad {} +
\mathbf{P} \biggl( \exp \biggl(\frac{a\xi}{\sigma} + \frac{a^2}{2\sigma^2} \biggr)
\leq\frac{d}s - 1 \biggr)
\nonumber
\\
&\qquad  = \biggl( \frac{d}s - 1 \biggr) \mathbf{P} \biggl( \xi>
\frac
{a}{2\sigma} + \frac{\sigma} a \log \biggl( \frac{d}s - 1 \biggr)
\biggr)\\
&\qquad \quad {} + \mathbf{P} \biggl( \xi\leq-\frac{a}{2\sigma} + \frac{\sigma} a \log
\biggl( \frac{d}s - 1 \biggr) \biggr)
\nonumber
\\
&\qquad  = \Psi_+(d,s,a).
\end{align*}
\end{proof}


%
%


\begin{proof}[Proof of Theorem~\ref{t1}]
For any $\theta\in\Theta_d(s,a)$, we have
%
\begin{equation}
\label{t1:1}
\begin{aligned}
\mathbf{E}_\theta \llvert \bar\eta- \eta \rrvert & =
\sum_{j: \theta_j=0} \mathbf{P}_{j,0}(\bar
\eta_j=1) +\sum_{j: \theta_j\ge a}
\mathbf{P}_{j,\theta_j}(\bar\eta _j=0)\\
&\quad {} + \sum
_{j: \theta_j\le-a} \mathbf{P}_{j,\theta_j}(\bar\eta_j=0)
\\
& \leq d\mathbf{P} \biggl( e^{- \frac{a^2}{2\sigma^2} }\cosh \biggl(\frac
{a\xi}{\sigma}
\biggr) > \frac{d}s - 1 \biggr)
\\
&\quad {}+ \sum_{j: \theta_j\ge a} \mathbf{P}_{j,\theta_j}(\bar
\eta_j=0) + \sum_{j: \theta_j\le-a}
\mathbf{P}_{j,\theta_j}(\bar\eta _j=0),
\end{aligned}
\end{equation}
where $\mathbf{P}_{j,\theta_j}$ denotes the distribution of $X_j$,
and $\xi
$ is a standard Gaussian random variable.
We now bound from above the probabilities $\mathbf{P}_{j,\theta
_j}(\bar\eta_j=0)$. Introduce the notation
\[
g(x) = \cosh \biggl(\frac{(x+ \sigma\xi)a }{\sigma^2} \biggr) \qquad\forall x\in\mathbb R,
\]
and
\[
u= \exp \biggl(\frac{a^2}{2 \sigma^2} + \log \biggl( \frac{d}s - 1 \biggr)
\biggr).
\]
We have
\[
\mathbf{P}_{j,\theta_j}(\bar\eta_j=0) = \mathbf{P}\bigl(g(\theta
_j)< u\bigr) = \mathbf{P} (-b-\theta_j < \sigma\xi< b-
\theta_j ),
\]
where $b= (\sigma^2/a)\operatorname{arccosh}(u)>0$. It is easy to check that
the function $x\mapsto\mathbf{P} (-b-x < \sigma\xi< b-x
)$ is
monotonically decreasing on $[0,\infty)$. Therefore, the maximum of
$\mathbf{P} (-b-\theta_j < \sigma\xi< b-\theta_j )$ over
$\theta_j\ge a$ is attained at $\theta_j=a$.
Thus, for any $\theta_j\ge a$ we have
%
\begin{equation}
\label{t1:2} \mathbf{P}_{j,\theta_j}(\bar\eta_j=0) \le
\mathbf{P}\bigl(g(a)< u\bigr) = \mathbf{P} \biggl(e^{- \frac{a^2}{2\sigma^2} }\cosh \biggl(
\frac{(a + \sigma\xi)a}{\sigma^2} \biggr)< \frac{d}s - 1 \biggr).
\end{equation}
Analogously, for any $\theta_j\le-a$,
%
\begin{equation}
\label{t1:3}
\begin{aligned}
\mathbf{P}_{j,\theta_j}(\bar\eta_j=0) & \le
\mathbf{P} \biggl(e^{- \frac
{a^2}{2\sigma^2} }\cosh \biggl(\frac{(-a + \sigma\xi)a }{\sigma
^2} \biggr)<
\frac{d}s - 1 \biggr)
\\
& = \mathbf{P} \biggl(e^{- \frac{a^2}{2\sigma^2} }\cosh \biggl(\frac
{(a +
\sigma\xi)a }{\sigma^2} \biggr)<
\frac{d}s - 1 \biggr),
\end{aligned}
\end{equation}
where the last equality follows from the fact that $\xi$ has the same
distribution as $-\xi$ and $\cosh$ is an even function.
Combining \eqref{t1:1}--\eqref{t1:3} proves the theorem.
\end{proof}


\begin{proof}[Proof of Theorem~\ref{thm:lb2}]
We follow the lines of the proof of Theorem~\ref{thm:lb} with suitable
modifications.
%
%

Let $\Theta^+$ and $\Theta^-$ be the sets of all $\theta$ in $\Theta
_d(s,a)$ such that $d-s$ components $\theta_j$ of $\theta$ are equal
to $0$ and the remaining $s$ components are equal to $a$ (for $\theta
\in\Theta^+$) or to $-a$ (for $\theta\in\Theta^-$). For any $\tilde \eta\in{\mathcal T}$, we have
\[
\begin{aligned}
&\sup_{\theta\in\Theta_d(s,a)} \sum_{j=1}^d
\mathbf{E}_{j,\theta_j} \llvert \tilde \eta_j - \eta_j
\rrvert \\
&\qquad \ge\frac{1}2 \Biggl\{\sup_{\theta\in\Theta^+} \sum
_{j=1}^d \mathbf{E}_{j,\theta_j} \llvert \tilde
\eta_j - \eta_j \rrvert + \sup_{\theta\in\Theta^-}
\sum_{j=1}^d \mathbf{E}_{j,\theta_j}
\llvert \tilde \eta_j - \eta _j \rrvert \Biggr\}.
\end{aligned}
\]
As shown in the proof of Theorem~\ref{thm:lb}, for any $\tilde\eta\in {\mathcal T}$,
\begin{eqnarray}
\sup_{\theta\in\Theta^+} \sum_{j=1}^d
\mathbf{E}_{j,\theta_j} \llvert \tilde \eta_j - \eta_j
\rrvert &\ge& \sum_{j=1}^d \biggl(
\biggl(1-\frac{s}d \biggr) {\mathbf{E}}_{j,0} (\tilde
\eta_j) + \frac{s}{d} {\mathbf{E}}_{j,a} (1-\tilde
\eta_j ) \biggr).
\nonumber
\end{eqnarray}
Analogously,\vspace*{2pt}
\begin{eqnarray}
\sup_{\theta\in\Theta^-} \sum_{j=1}^d
\mathbf{E}_{j,\theta_j} \llvert \tilde \eta_j - \eta_j
\rrvert &\ge& \sum_{j=1}^d \biggl(
\biggl(1-\frac{s}d \biggr) {\mathbf{E}}_{j,0} (\tilde
\eta_j) + \frac{s}{d} {\mathbf{E}}_{j,-a} (1-\tilde
\eta_j ) \biggr).
\nonumber
\end{eqnarray}
From the last three displays, we obtain\vspace*{3pt}
\begin{eqnarray*}
\sup_{\theta\in\Theta_d(s,a)} \sum_{j=1}^d
\mathbf{E}_{j,\theta_j} \llvert \tilde \eta_j - \eta_j
\rrvert &\ge& \sum_{j=1}^d \biggl(
\biggl(1-\frac{s}d \biggr) {\mathbf{E}}_{j,0} (\tilde
\eta_j) + \frac{s}{d} {\bar{\mathbf{E}}}_{j} (1-
\tilde \eta_j ) \biggr),
\end{eqnarray*}
where ${\bar{\mathbf{E}}}_j$ is the expectation with respect to the measure
${\bar{\mathbf{P}}}_j = (\mathbf{P}_{j,a}+\mathbf{P}_{j,-a})/2$. It
follows that\vspace*{3pt}
%
\begin{eqnarray}
\label{eq:lb2_2} \sup_{\theta\in\Theta_d(s,a)} \sum_{j=1}^d
\mathbf{E}_{j,\theta_j} \llvert \tilde \eta_j - \eta_j
\rrvert & \geq& \inf_{T \in[0,1]} \bigl( (d-s) {\mathbb
E}_0 (T)+ s {\bar{\mathbb E} }(1-T) \bigr). 
\end{eqnarray}
Here, ${\mathbb E}_0$ denotes the expectation with respect to the
distribution of $X$ with density $\varphi_{\sigma}(\cdot)$,
${\bar{\mathbb E}}$ is the expectation with respect to the
distribution of $X$ with mixture density $\bar\varphi_{\sigma}(\cdot)=(\varphi_{\sigma}(\cdot+a)+\varphi_{\sigma}(\cdot-a))/2$,
and
$\inf_{T \in[0,1]}$ denotes the infimum over all $[0,1]$-valued
statistics $T(X)$. Recall that we denote by $ \varphi_{\sigma}(\cdot
)$ is the density of ${\mathcal N}(0,\sigma^2)$ distribution. Set\vspace*{3pt}
\[
\tilde L=\inf_{T \in[0,1]} \biggl( \biggl(1-\frac{s}d
\biggr) {\mathbb E}_0 (T)+ \frac{s}{d} {\bar{\mathbb E}} (1-T)
\biggr).
\]
By the Bayesian version of the Neyman--Pearson lemma, the infimum here
is attained for $T=\tilde T$ given by\vspace*{3pt}
\[
\tilde T( X) = I \biggl( \frac{(s/d) \bar\varphi_{\sigma}
(X)}{(1-s/d) \varphi_{\sigma}(X)} >1 \biggr).
\]
Thus,\vspace*{3pt}
%
\begin{equation}
\label{eq:lb21}
\begin{aligned}
\tilde L &= \biggl(1-\frac{s}d \biggr) \mathbf{P}
\biggl( \frac{\bar\varphi_{\sigma} ({\sigma}\xi)}{ \varphi_{\sigma}({\sigma}\xi)} > \frac{d}s - 1 \biggr) + \frac{s}{2d} {
\mathbb P}_a \biggl( \frac{ \bar\varphi
_{\sigma} (X)}{ \varphi_{\sigma}(X)} \leq\frac{d}s - 1
\biggr)
\\[2pt]
&\quad {}+ \frac{s}{2d} {\mathbb P}_{-a} \biggl( \frac{\bar\varphi
_{\sigma} (X)}{ \varphi_{\sigma}(X)}
\leq\frac{d}s - 1 \biggr)
\\[2pt]
&= \biggl(1-\frac{s}d \biggr)\mathbf{P} \biggl(
e^{- \frac{a^2}{2\sigma^2}}\cosh \biggl(\frac{a\xi}{\sigma} \biggr) > \frac{d}s - 1
\biggr)
\\[2pt]
&\quad {}+ \frac{s}{2d} {\mathbb P}_a \biggl( \frac{ \bar\varphi
_{\sigma} (X)}{ \varphi_{\sigma}(X)}
\leq\frac{d}s - 1 \biggr) + \frac{s}{2d} {\mathbb P}_{-a}
\biggl( \frac{ \bar\varphi
_{\sigma} (X)}{ \varphi_{\sigma}(X)} \leq\frac{d}s - 1 \biggr),
\end{aligned}
\end{equation}
where ${\mathbb P}_u$ denotes the probability distribution of $X$ with
density $\varphi_{\sigma}(\cdot-u)$.
Note that, for all $x\in\mathbb R$,
\[
\frac{ \bar\varphi_{\sigma} (x)}{ \varphi_{\sigma}(x)} = e^{-
\frac{a^2}{2\sigma^2} }\cosh \biggl(\frac{ax }{\sigma^2} \biggr).
\]
Using this formula with $x=\sigma\xi+a$ and $x=\sigma\xi-a$, and
the facts that $\cosh(\cdot)$ is an even function and $\xi$
coincides with $-\xi$ in distribution, we obtain
\begin{eqnarray*}
{\mathbb P}_a \biggl( \frac{ \bar\varphi_{\sigma} (X)}{ \varphi
_{\sigma}(X)} \leq\frac{d}s - 1
\biggr) &=& {\mathbb P}_{-a} \biggl( \frac{ \bar\varphi_{\sigma} (X)}{ \varphi
_{\sigma}(X)} \leq
\frac{d}s - 1 \biggr)
\\
&=& \mathbf{P} \biggl( e^{- \frac{a^2}{2\sigma^2} }\cosh \biggl(\frac
{a\xi}{\sigma}+
\frac{a^2}{\sigma^2} \biggr) \le\frac{d}s - 1 \biggr).
\end{eqnarray*}
Thus, $\tilde L= (s/d)\bar\Psi(d,s,a)$. Combining this equality
with \eqref{eq:lb1} and \eqref{eq:lb2_2} proves the theorem.
\end{proof}


\begin{proof}[Proof of Theorem~\ref{cor:recovery_pattern}]
The upper bounds \eqref{cor:recovery_pattern:eq1}, \eqref{cor:recovery_pattern:eqbar2}
 and \eqref{cor:recovery_pattern:eq2}
follow immediately from \eqref{risks} and Theorems \ref{t2}, \ref{t1}
and \ref{nonasymptotic}, respectively.
We now prove the lower bound
\eqref{cor:recovery_pattern:eq3}.
To this end, first note that for any $\theta\in\Theta_d^+(s,a)$ and
any $\tilde\eta\in\mathcal T$ we have
\[
 \mathbf{P}_\theta\bigl(S_{\tilde\eta} \ne S(\theta)
\bigr)= \mathbf{P}_\theta\Biggl(\bigcup_{j=1}^d \{\tilde
\eta_j \ne\eta_j\}\Biggr) = 1-\prod
_{j=1}^d p_j(\theta),
\]
where $p_j(\theta)\triangleq\mathbf{P}_\theta(\tilde\eta_j =
\eta
_j)$. Hence, for any $\tilde\eta\in\mathcal T$,
%
\begin{equation}
\label{cor:recovery_pattern:eq5} \sup_{\theta\in\Theta_d^+(s,a)} \mathbf{P}_\theta
\bigl(S_{\tilde
\eta} \ne S(\theta)\bigr) \ge\max_{\theta\in\Theta'}
\mathbf{P}_\theta \bigl(S_{\tilde\eta} \ne S(\theta)\bigr) = 1-p_*,
\end{equation}
where $\Theta'$ is the subset of $\Theta_d^+(s,a)$ defined in the
proof of Theorem~\ref{thm:lb}, and $p_*= \min_{\theta\in\Theta'}
\prod_{j=1}^d p_j(\theta)$.

Next, for any selector $\tilde\eta$ we have $\mathbf{P}_\theta
(S_{\tilde\eta} \ne S(\theta))\ge\mathbf{P}_\theta( \llvert
\tilde\eta- \eta \rrvert =1)$. Therefore,\looseness=1
%
\begin{eqnarray}
\label{cor:recovery_pattern:eq5a} \sup_{\theta\in\Theta_d^+(s,a)} \mathbf{P}_\theta
\bigl(S_{\tilde
\eta} \ne S(\theta)\bigr) &\ge& \frac{1}{ \llvert \Theta' \rrvert } \sum
_{\theta\in\Theta
'} \mathbf{P}_{\theta} \bigl( \llvert \tilde\eta-
\eta \rrvert =1\bigr).
\end{eqnarray}
Here,
$\mathbf{P}_{\theta} ( \llvert \tilde\eta- \eta \rrvert
=1)=\mathbf{P}_{\theta} (\bigcup_{j=1}^d B_j)$ with the random events
$B_j=\{ \llvert \tilde\eta_j - \eta_j \rrvert =1,   \text{ and }   \tilde\eta_i = \eta_i,  \forall  i\ne j\}$. Since
the events $B_j$ are disjoint, for any $\tilde\eta\in\mathcal T$
we get\looseness=1
%
\begin{equation}\label{cor:recovery_pattern:eq6}
\begin{aligned}
& \frac{1}{ \llvert \Theta' \rrvert } \sum_{\theta\in
\Theta'}
\mathbf{P}_{\theta} \bigl( \llvert \tilde\eta- \eta \rrvert =1\bigr)\\
&\qquad  =
\frac{1}{ \llvert \Theta' \rrvert } \sum_{\theta\in\Theta
'} \sum
_{j=1}^d \mathbf{P}_{\theta}
(B_j)
\\
&\qquad = \frac{1}{ \llvert \Theta' \rrvert } \sum_{j=1}^d
\biggl(\sum_{\theta\in\Theta':\theta_j=0} {\mathbf{P}}_{j,0} (
\tilde \eta_j=1) \prod_{i\ne j}
p_i(\theta)\\
&\qquad \quad {} + \sum_{\theta\in\Theta
':\theta_j=a} {
\mathbf{P}}_{j,a} (\tilde\eta_j=0) \prod
_{i\ne j} p_i(\theta) \biggr)
\\
&\qquad \ge \frac{p_*}{ \llvert \Theta' \rrvert } \sum_{j=1}^d
\biggl(\sum_{\theta\in\Theta':\theta_j=0} {\mathbf{P}}_{j,0} (
\tilde \eta_j=1) + \sum_{\theta\in\Theta':\theta_j=a} {
\mathbf{P}}_{j,a} (\tilde\eta_j=0) \biggr)
\\
&\qquad =  \frac{p_*}{ \llvert \Theta' \rrvert } \sum_{j=1}^d
\biggl(\sum_{\theta\in\Theta':\theta_j=0} {\mathbf{E}}_{j,0} (
\tilde \eta_j) + \sum_{\theta\in\Theta':\theta_j=a} {
\mathbf{E}}_{j,a} (1-\tilde\eta_j) \biggr),
\end{aligned}
\end{equation}
where ${\mathbf{P}}_{j,u}$ denotes the distribution of $X_j$ when
$\theta_j=u$.
We now bound the right-hand side of \eqref{cor:recovery_pattern:eq6}
by following the argument from the last three lines of \eqref{eq:lb2}
to the end of the proof of Theorem~\ref{thm:lb}. Applying this
argument yields that, for any $\tilde\eta\in\mathcal T$,
%
\begin{eqnarray}
\label{cor:recovery_pattern:eq7} \frac{1}{ \llvert \Theta' \rrvert } \sum_{\theta\in\Theta
'}
\mathbf{P}_{\theta} \bigl( \llvert \tilde\eta- \eta \rrvert =1\bigr)&\ge&
p^*d \tilde L \ge p^*s \Psi_+(d,s,a).
\end{eqnarray}
Combining \eqref{cor:recovery_pattern:eq5}, \eqref
{cor:recovery_pattern:eq5a} and \eqref{cor:recovery_pattern:eq7}, we
find that, for any $\tilde\eta\in\mathcal T$,
\[
\begin{aligned}
\sup_{\theta\in\Theta_d^+(s,a)} \mathbf{P}_\theta\bigl(S_{\tilde
\eta}
\ne S(\theta)\bigr) &\ge\min_{0\le p^*\le1}\max\bigl\{1-p^*, p^*s \Psi
_+(d,s,a)\bigr\}\\
&=\frac{s\Psi_+(d,s,a)}{1+s\Psi_+(d,s,a)}.
\end{aligned}
\]

We now prove the lower bound \eqref{cor:recovery_pattern:eqbar3}. Let
the sets $\Theta^+$ and $\Theta^-$ and the constants $p_j(\theta)$
be the same as in the proof of Theorem~\ref{thm:lb2}. Then
\begin{equation*}
\label{eq:newlb1} \sup_{\theta\in\Theta_d(s,a)} \mathbf{P}_\theta
\bigl(S_{\tilde
\eta} \ne S(\theta)\bigr) \ge\max_{\theta\in\Theta^+ \cup\Theta^-}
\mathbf {P}_\theta \bigl(S_{\tilde\eta} \ne S(\theta)\bigr) = 1-\bar p,
\end{equation*}
where $\bar p = \min_{\theta\in\Theta^+ \cup\Theta^-}\prod_{j=1}^d p_j(\theta)$.

For any selector $\tilde\eta$, we use that $\mathbf{P}_\theta
(S_{\tilde\eta} \ne S(\theta))\ge\mathbf{P}_\theta( \llvert
\tilde\eta- \eta \rrvert =1)$ and, therefore,
\[
\begin{aligned}
\sup_{\theta\in\Theta_d(s,a)} \mathbf{P}_\theta\bigl(S_{\tilde
\eta}
\ne S(\theta)\bigr) &\ge \frac{1}{2 \llvert \Theta^+ \rrvert } \sum_{\theta\in\Theta
^+}
\mathbf{P}_{\theta} \bigl( \llvert \tilde\eta- \eta \rrvert =1\bigr) \\
&\quad {}+
\frac{1}{2 \llvert \Theta^- \rrvert } \sum_{\theta\in
\Theta^-} \mathbf{P}_{\theta}
\bigl( \llvert \tilde\eta- \eta \rrvert =1\bigr).
\end{aligned}
\]
We continue along the same lines as in the proof of \eqref
{cor:recovery_pattern:eq6} to get, for any separable selector $\tilde
\eta$,
\[
\begin{aligned}
&\sup_{\theta\in\Theta_d(s,a)} \mathbf{P}_\theta\bigl(S_{\tilde
\eta}
\ne S(\theta)\bigr)\\
&\qquad \ge \frac{\bar p}{2  \llvert \Theta^+ \rrvert } \sum_{j=1}^d
\biggl( \sum_{\theta\in\Theta^+: \theta_j=0} \mathbf {E}_{j,0}(
\tilde \eta_j) + \sum_{\theta\in\Theta^+: \theta_j=a} \mathbf{E}
_{j,a}(1-\tilde\eta_j) \biggr)
\\
&\qquad \quad {}+ \frac{\bar p}{2  \llvert \Theta^- \rrvert } \sum_{j=1}^d
\biggl( \sum_{\theta\in\Theta^-: \theta_j=0} \mathbf{E} _{j,0}(\tilde
\eta_j) + \sum_{\theta\in\Theta^-: \theta_j=-a}
\mathbf{E}_{j,-a}(1-\tilde\eta_j) \biggr)
\\
&\qquad  \geq \frac{\bar p}2 \sum_{j=1}^d
\biggl( \biggl( 1- \frac{s}d \biggr) \mathbf{E}_{j,0} (\tilde
\eta_j) + \frac{s}d \mathbf{E}_{j, a}(1 - \tilde
\eta_j) \biggr)
\\
&\qquad \quad {}+ \frac{\bar p}2 \sum_{j=1}^d
\biggl( \biggl( 1- \frac{s}d \biggr) \mathbf{E}_{j,0} (\tilde
\eta_j) + \frac{s}d \mathbf{E}_{j, -a}(1 - \tilde
\eta_j) \biggr)
\\
&\qquad =  \bar p \sum_{j=1}^d \biggl( \biggl(
1- \frac{s}d \biggr) \mathbf{E} _{j,0} (\bar
\eta_j) + \frac{s}d \bar{\mathbf{E}}_{j}(1 -
\tilde\eta_j) \biggr),
\end{aligned}
\]
where again $\bar{\mathbf{E}}_j$ denotes the expected value with
respect to
$\bar{\mathbf{P}}_j = \frac{1}2 (\mathbf{P}_{j,a}+ \mathbf{P}_{j,
-a})$. Analogously to the
proof of Theorem~\ref{thm:lb2}, the expression in the last display can
be further bounded from below by
$\bar p d \tilde L = \bar p s \bar\Psi(d,s,a)$. Thus,
\[
\begin{aligned}
\sup_{\theta\in\Theta_d(s,a)} \mathbf{P}_\theta\bigl(S_{\tilde
\eta}
\ne S(\theta)\bigr) &\ge\min_{0\le\bar p \le1}\max\bigl\{1-\bar p, \bar p s
\bar\Psi(d,s,a)\bigr\}\\
&=\frac{s \bar\Psi(d,s,a)}{1+s \bar\Psi(d,s,a)}.
\end{aligned}
\]
\end{proof}

\begin{proof}[Proof of Theorem~\ref{t4}]
(i) It follows from the second inequality in \eqref{term1} that
%
\begin{equation}
\label{xx} \sup_{\theta\in\Theta_d(s,a)} \frac{1}s
\mathbf{E}_\theta \llvert \hat \eta- \eta \rrvert \leq 2 \biggl(
\frac{d}s - 1 \biggr) \Phi(-t/\sigma) + 2 \Phi \bigl(-(a-t)_+/\sigma
\bigr),
\end{equation}
where $t= \frac{a}2 + \frac{\sigma^2}a \log ( \frac{d}s -
1 )$ is the threshold (\ref{threshold}).
Since $a^2 \geq2 \sigma^2\log(d/s - 1)$ we get that $a\ge t$ and
that $t> a/2$, which is equivalent to $t > a-t$.
Furthermore, $ (\frac{d}s -1 )e^{-t^2/(2\sigma^2)} =
e^{-(a-t)^2/(2\sigma^2)} $. These remarks and \eqref{AS} imply that
\begin{eqnarray*}
\biggl(\frac{d}s - 1 \biggr) \Phi(-t/\sigma) &\leq& \sqrt{
\frac{2}\pi} \frac{\exp(-(a-t)^2/(2\sigma
^2))}{(a-t)/\sigma+ \sqrt{(a-t)^2/\sigma^2 + 8/\pi}}
\\
&\leq& \frac{\exp(-(a-t)^2/(2\sigma^2))}{(a-t)/\sigma+ \sqrt
{(a-t)^2/\sigma^2 + 4}}
\\
&\leq& \sqrt{\frac{\pi}2} \Phi \biggl(- \frac{a-t}{\sigma} \biggr).
\end{eqnarray*}
Combining this with \eqref{xx}, we get
\begin{eqnarray*}
\sup_{\theta\in\Theta_d(s,a)} \frac{1}s \mathbf{E}_\theta
\llvert \hat \eta- \eta \rrvert &\leq& ( 2 + \sqrt{2 \pi}) \Phi \biggl(-
\frac{a-t}{\sigma} \biggr).
\end{eqnarray*}
Now, to prove \eqref{eq:t4:1} it remains to note that under assumption
\eqref{ass:e},
\[
\frac{a-t}\sigma= \frac{a}{2\sigma} - \frac{\sigma}a \log \biggl(
\frac{d}s - 1 \biggr)= \frac{a^2-2\sigma^2\log((d-s)/s)}{2a
\sigma}\ge\Delta.
\]
Indeed, assumption \eqref{ass:e} states that $a\ge a_0\triangleq
\sigma (2\log((d-s)/s) + W )^{1/2}$, and
the function $a\mapsto ({a^2-2\sigma^2\log((d-s)/s)}
)/{a}$ is monotonically increasing in $a>0$. On the other hand,
%
\begin{equation}
\label{eq:t4:3} \bigl( {a_0^2-2\sigma^2\log
\bigl((d-s)/s\bigr)} \bigr)/(2{a_0}\sigma) = \Delta.
\end{equation}

(ii) We now prove \eqref{eq:t4:2}. By Theorem~\ref{thm:lb},
\[
\inf_{\tilde\eta} \sup_{\theta\in\Theta_d(s,a)} \frac{1}s
\mathbf{E} _\theta \llvert \tilde\eta- \eta \rrvert \geq
\frac{s'}s \Phi \biggl( - \frac{a} {2\sigma} + \frac{\sigma} a \log \biggl(
\frac{d}s - 1 \biggr) \biggr) - 4 \frac{s'}s \exp\left(-\frac{(s-s')^2}{2s} \right).
\]
Here,
\[
- \frac{a} {2\sigma} + \frac{\sigma} a \log \biggl( \frac{d}s - 1
\biggr)= \frac{2\sigma^2\log((d-s)/s)-a^2}{2\sigma a}.
\]
Observe that
the function $a\mapsto ({2\sigma^2\log((d-s)/s)-a^2}
)/{a}$ is monotonically decreasing in $a>0$
and that assumption (\ref{lowE}) states that $a\le a_0$. In view of
\eqref{eq:t4:3}, the value of its minimum for $a\le a_0$ is
equal to $- \Delta$.
The bound \eqref{eq:t4:2} now follows by the monotonicity of $\Phi
(\cdot)$.
\end{proof}

\begin{proof}[Proof of Theorem~\ref{thm:asymp}]
Assume without loss of
generality that $d$ is large enough to have $(d-s_d)/s_d>1$. We apply
Theorem~\ref{t4} with
$
W= A \sqrt{2\log((d-s_d)/s_d)}$.
Then
\[
\Delta^2= \frac{A^2 \sqrt{2\log((d-s_d)/s_d) }}{4 (\sqrt{2\log
((d-s_d)/s_d)}+A )}.
\]
By assumption, there exists $\nu>0$ such that $(2+\nu)s_d \le d$ for
all $d$ large enough. Equivalently, $d/s_d - 1 \geq1+\nu$ and,
therefore, using the monotonicity argument, we find
\[
\Delta^2\ge\frac{A^2 \sqrt{2\log(1+\nu) }}{\sqrt{2\log(1+\nu
)}+A}\to\infty\qquad\mbox{as } A\to\infty.
\]
This and \eqref{eq:t4:1} imply part (i) of the theorem.

Part (ii)
follows from \eqref{eq:t4:2} by noticing that 
$\Delta^2\le\sup_{x>0}\frac{ A^2x}{4(x+A)}=A^2/4$ for any fixed $A>0$.
Now, for $s$ large enough, let us put $s'=(1-\varepsilon)s$ for some $\varepsilon$ in (0,1), fixed. Thus, the lower bound of the risk becomes
\begin{equation*}
(1-\varepsilon) \Phi(-\Delta) - 4 \exp \left( -\frac s2 (1-\varepsilon)^2\right) >0,
\end{equation*}
for $s$ large enough.
\end{proof}

\begin{proof}[Proof of Theorem~\ref{thm:asympE}]
Throughout the proof, we assume without loss of generality that $d$ is
large enough to have $s_d\ge2$, and $(d-s_d)/s_d>1$.
Set $W_*(s)\triangleq4  (\log s + \sqrt{\log s \log(d-s)}
)$, and notice that
%
\begin{eqnarray}
\label{eq3a:thm:asympE} \frac{W_*(s_d)}{2 \sqrt{2 \log((d-s_d)/s_d) +W_*(s_d)}}&=&\sqrt {2\log s_d},
\\
\label{eq3b:thm:asympE} 2 \log\bigl((d-s_d)/s_d\bigr)
+W_*(s_d) &=& 2 \bigl(\sqrt{\log(d-s_d)} + \sqrt {\log
s_d} \bigr)^2.
\end{eqnarray}
If \eqref{eq1:thm:asympE} holds, we have $W_d\ge W_*(s_d)$ for all $d$
large enough. By the monotonicity of the quantity $\Delta$ defined in
(\ref{DeltaERec})
with respect to $W$, this implies
%
\begin{equation}
\label{eq4:thm:asymp}
\begin{aligned}
\Delta_d &\triangleq \frac{W_d}{2 \sqrt{2 \log((d-s_d)/s_d) +W_d}}
\\
&\ge\frac{W_*(s_d)}{2 \sqrt{2 \log((d-s_d)/s_d) +W_*(s_d)}}=\sqrt {2\log s_d} .
\end{aligned}
\end{equation}
Now, by Theorem~\ref{t4} and using (\ref{AS}) we may write
\begin{equation}\label{eq3:thm:asymp}
\begin{aligned}
\sup_{\theta\in\Theta_d(s_d,a_d)} \mathbf{E}_\theta \llvert \hat
\eta - \eta \rrvert &\leq (2+\sqrt{2\pi})s_d \Phi (-
\Delta_d )
\\
&\leq 3s_d \min \biggl\{1, \frac{1}{\Delta_d} \biggr\} \exp
\biggl(- \frac{\Delta_d^2}2 \biggr)
\\
 &= 3 \min \biggl\{1, \frac{1}{\Delta_d} \biggr\} \exp \biggl(-
\frac{\Delta_d^2-2\log s_d}2 \biggr).
\end{aligned}
\end{equation}
This and \eqref{eq4:thm:asymp} imply that, for all $d$ large enough,
\[
\sup_{\theta\in\Theta_d(s_d,a_d)} \mathbf{E}_\theta \llvert \hat\eta - \eta
\rrvert \le3 \min \biggl\{1, \frac{1}{\sqrt{2\log
s_d}} \biggr\}.
\]
Since $s_d\to\infty$, part (i) of the theorem follows.

We now prove part (ii) of the theorem. It suffices to consider $W_d>
0$ for all $d$ large enough since for nonpositive $W_d$ almost full
recovery is impossible and the result follows from part (ii) of
Theorem~\ref{thm:asymp}. If \eqref{eq2:thm:asymp} holds, there exists
$A<1$ such that $W_d\le AW_*(s_d)$ for all $d$ large enough. By the
monotonicity of the quantity $\Delta$ defined in (\ref{DeltaERec})
with respect to $W$ and in view of equation \eqref{eq3a:thm:asympE},
this implies
\begin{equation}\label{eq5:thm:asymp}
\begin{aligned}
& \Delta_d^2 - 2\log s_d
\\
&\qquad \le \frac{A^2W_*^2(s_d)}{4 (2 \log((d-s_d)/s_d) +AW_*(s_d))}\\
&\qquad \quad {}- \frac{W_*^2(s_d)}{4 (2 \log((d-s_d)/s_d) +W_*(s_d))}
\\
&\qquad = \frac{(A-1)W_*^2(s_d)(AW_*(s_d)+2 (A+1) \log((d-s_d)/s_d))}{4 (2
\log((d-s_d)/s_d) +AW_*(s_d))(2 \log((d-s_d)/s_d) +W_*(s_d))}
\\
&\qquad \le \frac{(A-1)AW_*^2(s_d)}{4 (2 \log((d-s_d)/s_d) +W_*(s_d))}
\\
&\qquad = \frac{2(A-1)A  (\log s_d + \sqrt{\log s_d \log(d-s_d)}
)^2}{ (\sqrt{\log(d-s_d)} + \sqrt{\log s_d}  )^2} = 2(A-1)A \log s_d,
\end{aligned}
\end{equation}
where we have used the fact that $A<1$ and equations \eqref
{eq3a:thm:asympE}, \eqref{eq3b:thm:asympE}.
Next, by Theorem~\ref{t4} and using (\ref{AS}), we have for $s'=s_d/2$,
\begin{eqnarray}
\nonumber
\inf_{\tilde{\eta}} \sup_{\theta\in\Theta_d(s_d,a_d)} \mathbf{E}
_\theta \llvert \tilde\eta- \eta \rrvert &\ge& \frac{s_d}2 \left( \Phi (-\Delta_d ) - 4 \exp \left(-\frac{s_d}8 \right) \right)
\end{eqnarray}
and
\begin{eqnarray}
\nonumber
\frac{s_d}2  \Phi (-\Delta_d ) & \ge & \frac{s_d}8 \min \biggl\{\frac{1}2,
\frac{1}{\Delta_d} \biggr\} \exp \biggl(- \frac{\Delta_d^2}2 \biggr)
\\
\nonumber
&=& \frac{1}8 \min \biggl\{\frac{1}2,
\frac{1}{\Delta_d} \biggr\} \exp \biggl(- \frac{\Delta_d^2-2\log s_d}2 \biggr).
\end{eqnarray}
Combining this inequality with \eqref{eq5:thm:asymp}, we find that,
for all $d$ large enough,
\[
\inf_{\tilde{\eta}}\sup_{\theta\in\Theta_d(s_d,a_d)} \mathbf{E}
_\theta \llvert \tilde\eta- \eta \rrvert \ge\frac{1}8 \min
\biggl\{\frac{1}2, \frac{1}{\Delta_d} \biggr\} e^{ (1-A)A \log
s_d } - 2 s_d e^{-s_d/8}.
\]
Since $A<1$ and $\Delta_d\le A\sqrt{2\log s_d}$ by \eqref
{eq5:thm:asymp}, the last expression tends to $\infty$ as $s_d\to
\infty$. This proves part (ii) of the theorem.
\end{proof}

\begin{proof}[Proof of Theorem~\ref{thm:adaptE}]
By \eqref{term1}, for any $\theta\in\Theta_d(s_d,a_d)$, and any
$t>0$ we have
\begin{equation*}
\label{term1**} \mathbf{E}_\theta \llvert \hat\eta- \eta \rrvert \leq d \mathbf{P}\bigl( \llvert \xi \rrvert \geq t/\sigma\bigr) +
s_d \mathbf{P}\bigl( \llvert \xi \rrvert >(a_d- t)_+/
\sigma\bigr),
\end{equation*}
where $\xi$ is a standard normal random variable. It follows that, for
any $a_d\ge a_d^*$, any $\theta\in\Theta_d(s_d,a_d)$, and any $t>0$,
\begin{equation*}
\label{term1*} \mathbf{E}_\theta \llvert \hat\eta- \eta \rrvert \leq d
\mathbf{P}\bigl( \llvert \xi \rrvert \geq t/\sigma\bigr) + s_d
\mathbf{P}\bigl( \llvert \xi \rrvert >\bigl(a_d^*- t\bigr)_+/\sigma
\bigr).
\end{equation*}
Without loss of generality assume that $d\ge6$ and $2\le s_d\le d/2$.
Then, using the inequality $\sqrt{x}-\sqrt{y} \le(x-y)/\sqrt{2y}$,
$\forall x>y>0$, we find that, for $t=\sigma\sqrt{2 \log d}$,
\begin{eqnarray*}
\bigl(a_d^* -t\bigr)_+/\sigma&\geq& \sqrt{2} \bigl(\sqrt{
\log(d-s_d)} -\sqrt{\log d }+ \sqrt{\log(s_d)} \bigr)
\\
&\geq& \sqrt{2\log(s_d)} - \log \biggl(\frac{d}{d-s_d} \biggr)/
\sqrt{\log (d-s_d)}
\\
&\geq& \sqrt{2\log(s_d)} - (\log2)/\sqrt{\log(d/2)} >0.
\end{eqnarray*}
From this we also easily deduce that, for $2\le s_d\le d/2$, we have
$((a_d^* -t)_+/\sigma)^2/2\ge\log(s_d) - \sqrt{2}\log2$.
Combining these remarks with \eqref{AS} and \eqref{phase1}, we find
\begin{eqnarray*}
\sup_{\theta\in\Theta_d(s_d,a_d)} \mathbf{E}_\theta \llvert \hat\eta - \eta
\rrvert &\leq& \frac{1}{\sqrt{2\log d }}+ \frac{s_d \exp (-\log(s_d) + \sqrt{2}\log2 )} {\sqrt
{2\log(s_d)}},
\end{eqnarray*}
which immediately implies the theorem by taking the limit as $d\to
\infty$.
\end{proof}


\begin{proof}[Proof of Theorem~\ref{th:adapt}]
Throughout the proof, we will write for brevity $s_d=s, a_d=a, A_d=A$,
and set $\sigma=1$. Since $\Theta_d(s,a)\subseteq\Theta
_d(s,a_0(s,A))$ for all $a\ge a_0(s,A)$, it suffices to prove that
%
\begin{equation}
\label{eq:th:adapt1} \lim_{d\to\infty} \sup_{\theta\in\Theta_d(s,a_0(s,A))}
\frac{1}{s} \mathbf{E}_\theta \bigl\llvert \hat
\eta^{\mathrm{ad}}- \eta \bigr\rrvert = 0.
\end{equation}
Here, $s\le s^*_d$ and recall that throughout this section we assume
that $s_d^*\le d/4$;
since we deal with asymptotics as $d/s^*_d\to\infty$,
the latter assumption is without loss of generality in the current proof.

If $s<g_M$, let $m_0\in\{2,\dots,M\}$ be the index such that
$g_{m_0}$ is the minimal element of the grid, which is greater than the
true underlying $s$. Thus, $g_{m_0}/2=g_{m_0-1} \leq s < g_{m_0}$. If
$s\in[g_M,s^*_d]$, we set $m_0=M$. In both cases,
%
\begin{equation}
\label{eq:th:adapt1_bis} s\ge g_{m_0}/2.
\end{equation}
We decompose the risk as follows:
\[
\frac{1}s \mathbf{E}_\theta \bigl\llvert \hat
\eta^{\mathrm{ad}}- \eta \bigr\rrvert = I_1 + I_2,
\]
where
\begin{eqnarray*}
I_1 & = & \frac{1}s \mathbf{E}_\theta \bigl( \bigl
\llvert \hat\eta (g_{\hat{m}}) - \eta \bigr\rrvert I(\hat{m} \leq
m_0) \bigr),
\\
I_2 &=& \frac{1}s \mathbf{E}_\theta \bigl( \bigl
\llvert \hat\eta (g_{\hat
{m}}) - \eta \bigr\rrvert I( \hat{m} \ge
m_0+1) \bigr).
\end{eqnarray*}
We now evaluate $I_1$. Using the fact that $\hat\eta_j(g_{m})$ is
monotonically increasing in $m$ and the definition of $\hat{m}$,
we obtain that, on the event $\{\hat{m} \leq m_0\}$,
\begin{eqnarray*}
\bigl\llvert \hat\eta(g_{\hat{m}}) - \hat\eta(g_{m_0}) \bigr
\rrvert &\le& \sum_{m=\hat{m} +1}^{m_0} \bigl\llvert
\hat\eta (g_{m}) - \hat\eta(g_{m-1}) \bigr\rrvert
\\
&=& \sum_{m=\hat{m}+1}^{m_0} \sum
_{j=1}^{d}\bigl(\hat\eta_j(g_{m})
- \hat\eta_j(g_{m-1}) \bigr)
\\
&=& \sum_{m=\hat{m}+1}^{m_0} \sum
_{j=1}^d I \bigl(w(g_m)\le \llvert
X_j \rrvert < w(g_{m-1}) \bigr)
\\
&\le&\tau\sum_{m=\hat{m}+1}^{m_0} g_m
\le{\tau} s \sum_{m=2}^{m_0} 2^{m-m_0+1}
\le4\tau s,
\end{eqnarray*}
where we have used the equality $g_m = 2^m$  and \eqref{eq:th:adapt1_bis}.
Thus,
%
\begin{equation}
\label{term10}
\begin{aligned}
I_1 & \le \frac{1}s \mathbf{E}_\theta
\bigl( \bigl\llvert \hat\eta (g_{\hat{m}}) - \hat\eta(g_{m_0})
\bigr\rrvert I(\hat{m} \le m_0) \bigr) + \frac{1}s
\mathbf{E}_\theta \bigl\llvert \hat\eta(g_{m_0}) - \eta \bigr
\rrvert
\\
& \leq 4\tau+ \frac{1}s \mathbf{E}_\theta \bigl\llvert \hat
\eta (g_{m_0}) - \eta \bigr\rrvert .
\end{aligned}
\end{equation}
Next, note that the first inequality in \eqref{term1} is true for any $t>0$. Applying it with $t=w(g_{m_{0}})$, we obtain
%
\begin{equation}
\label{term11}
\begin{aligned}
\frac{1}s \mathbf{E}_\theta \bigl\llvert \hat
\eta(g_{m_0}) - \eta \bigr\rrvert &\leq \frac{d}s \mathbf{P}\bigl( \llvert \xi \rrvert \geq w(g_{m_0})\bigr)\\
&\quad {} +
\mathbf{P}\bigl( \llvert \xi \rrvert >\bigl(a_0(s,A)-
w(g_{m_0})\bigr)_+\bigr)
\end{aligned}
\end{equation}
where $\xi$ is a standard Gaussian random variable. Using the bound on
the Gaussian tail probability and the fact that $g_{m_0} > s\ge g_{m_0}/2$, we get
%
\begin{equation}
\label{term12}
\begin{aligned}
 \frac{d}s  \mathbf{P}\bigl( \llvert
\xi \rrvert \geq w(g_{m_0})\bigr)&\le\frac{d/s}{d/g_{m_0}-1}
\frac{\pi^{-1/2}}{\sqrt
{\log(d/g_{m_0} - 1)}}
\\
&\le \frac{d}{d-2s} \frac{2\pi^{-1/2}}{\sqrt{\log(d/s - 1)}} \le \frac{4\pi^{-1/2}}{\sqrt{\log(d/s^*_d - 1)}}.
\end{aligned}
\end{equation}
To bound the second probability on the right-hand side of \eqref
{term11}, we use the following lemma.
%
\begin{lemma}\label{lem1} Under the assumptions of Theorem~\ref
{th:adapt}, for any $m\ge m_0$ we have
%
\begin{equation}
\label{eq:lem1} \mathbf{P}\bigl( \llvert \xi \rrvert >\bigl(a_0(s,A)-
w(g_{m})\bigr)_+\bigr)\le \bigl(\log \bigl({d}/{s^*_d} -1
\bigr) \bigr)^{-\frac{1}{2}}.
\end{equation}
\end{lemma}

Combining \eqref{term11}, \eqref{term12} and \eqref{eq:lem1} with
$m=m_0$, we find
%
\begin{eqnarray}
\label{term141} \frac{1}s \mathbf{E}_\theta \bigl\llvert \hat
\eta(g_{m_0}) - \eta \bigr\rrvert & \leq& \frac{4\pi^{-1/2}+1}{\sqrt{\log(d/s^*_d - 1)}} ,
\end{eqnarray}
which together with \eqref{term10} leads to the bound
%
\begin{eqnarray}
\label{term14} I_1 & \leq& 4\tau+ \frac{4\pi^{-1/2}+1}{\sqrt{\log(d/s^*_d - 1)}} .
\end{eqnarray}
We now turn to the evaluation of $I_2$. {It is enough to consider the
case $m_0\le M-1$ since $I_2=0$ when $m_0=M$. }We have
%
\begin{equation}
\label{term15}
\begin{aligned}
 I_2 &= \frac{1}s \sum
_{m=m_0+1}^M \mathbf{E}_\theta \bigl( \bigl
\llvert \hat \eta(g_{\hat{m}}) - \eta \bigr\rrvert I( \hat{m} = m)
\bigr)
\\
& \leq \frac{1}s \sum_{m=m_0+1}^M
\bigl(\mathbf{E}_\theta \bigl\llvert \hat \eta(g_{m}) - \eta
\bigr\rrvert^{2} \bigr)^{1/2} \bigl(\mathbf{P}_\theta (\hat{m} =
m) \bigr)^{1/2}.
\end{aligned}
\end{equation}
By definition, the event $\{\hat{m} = m\}$ occurs implies that $\sum_{j=1}^d I( w_m\le
 \llvert  X_j \rrvert < w_{m-1} ) > \tau g_m\triangleq
v_m$, where we set for brevity $w_m=w(g_{m})$. Thus,
%
\begin{eqnarray}
\mathbf{P}_\theta(\hat{m} = m) & \leq&  \mathbf{P}_\theta \Biggl(\sum
_{j=1}^d I\bigl( w_m\le
\llvert X_j \rrvert < w_{m-1} \bigr) > v_m
\Biggr). \label{I2}
\end{eqnarray}
By Bernstein's inequality, for any $t>0$ we have
%
\begin{equation}
\label{I2B}
\begin{aligned}
& \mathbf{P}_\theta \Biggl(\sum
_{j=1}^d I\bigl( w_m\le \llvert
X_j \rrvert < w_{m-1} \bigr) - \mathbf{E}_\theta
\Biggl( \sum_{j=1}^d I\bigl(
w_m\le \llvert X_j \rrvert < w_{m-1} \bigr)
\Biggr) > t \Biggr)
\\
&\qquad\leq\exp \biggl( - \frac{t^2/2}{ \sum_{j=1}^d \mathbf
{E}_\theta
 ( I( w_m\le \llvert  X_j \rrvert < w_{m-1})
) +2t/3 } \biggr),
\end{aligned}
\end{equation}
where we have used that, for random variables with values in $\{0,1\}$,
the variance is smaller than the expectation.

Now, similar to \eqref{term1}, for any $\theta\in\Theta_d(s,a_0(s,A))$,
\[
\begin{aligned}
&\mathbf{E}_\theta \Biggl( \sum_{j=1}^d
I\bigl(w_m\le \llvert X_j \rrvert < w_{m-1}
\bigr) \Biggr)
\\
&\qquad \le d \mathbf{P} \bigl( w_m \le \llvert \xi \rrvert <
w_{m-1} \bigr) + \sum_{j:\theta_j\ne0}\mathbf{P}
\bigl( \llvert \theta_j+ \xi \rrvert < w_{m-1} \bigr)
\\
&\qquad  \leq d \mathbf{P} \bigl( \llvert \xi \rrvert \ge w_{m}
\bigr)+ s \mathbf{P}\bigl( \llvert \xi \rrvert > -\bigl(a_0(s,A) -
w_{m-1}\bigr)_+\bigr),
\end{aligned}
\]
where $\xi$ is a standard Gaussian random variable. Since $m\ge
m_0+1$, from Lemma~\ref{lem1} we get
\begin{equation}
\label{term17} \mathbf{P}\bigl( \llvert \xi \rrvert >\bigl(a_0(s,A)-
w_{m
-1}\bigr)_+\bigr)\le \bigl( \log \bigl({d}/{s^*_d} -1
\bigr) \bigr)^{-\frac{1}{2}}.
\end{equation}
Next, using the bound on the Gaussian tail probability and the
inequalities $g_m\le s^*_d\le d/4$, we find
\begin{equation}
\label{term18} d \mathbf{P} \bigl( \llvert \xi \rrvert \ge w_{m}
\bigr)\le\frac{d}{d/g_m- 1} \frac{\pi^{-1/2}}{\sqrt{\log
(d/g_m- 1)}}\le\frac{(4/3)\pi^{-1/2}g_m}{\sqrt{\log
(d/s^*_d - 1)}}.
\end{equation}
We now deduce from \eqref{term17} and \eqref{term18}, and the
inequality $s\le g_m$ for $m\ge m_0+1$, that
%
\begin{equation}
\label{I221} \mathbf{E}_\theta \Biggl( \sum
_{j=1}^d I\bigl( w_m\le \llvert
X_j \rrvert < w_{m-1}\bigr) \Biggr) \leq
\frac{ ((4/3)\pi^{-1/2}+1 )g_m}{\sqrt{\log(d/s^*_d
- 1)}} \le2 \tau g_m.
\end{equation}
Taking in \eqref{I2B} $t=3\tau g_m=3v_m$ and using \eqref
{I221}, we find
\begin{equation*}
\label{term19} \mathbf{P}_\theta \Biggl(\sum
_{j=1}^d I\bigl( w_m\le \llvert
X_j \rrvert < w_{m-1} \bigr) > v_m \Biggr)
\leq\exp(-C_1 v_m)=\exp\bigl(-C_1
2^m\tau\bigr),
\end{equation*}
for some absolute constant $C_1>0$. This implies
%
\begin{equation}
\label{term20} \mathbf{P}_\theta(\hat{m} = m)\le\exp\bigl(-C_1
2^m \tau\bigr).
\end{equation}

On the other hand, notice that the bounds \eqref{term11}, and \eqref
{term12} are valid not only for $g_{m_0}$ but also for any $g_{m}$ with
$m\ge m_0+1$. Using this observation and Lemma~\ref{lem1} we get that,
for any $\theta\in\Theta_d(s,a_0(s,A))$ and any $m\ge m_0+1$,
%
\begin{equation}\label{term241}
\begin{aligned}
\mathbf{E}_\theta \bigl\llvert \hat\eta(g_{m}) - \eta \bigr
\rrvert & \leq s \biggl[\frac{d/s}{d/g_{m}-1} \frac{\pi^{-1/2}}{\sqrt{\log
(d/g_{m} - 1)}} + \bigl( \log
\bigl({d}/{s^*_d} -1 \bigr) \bigr)^{-\frac{1}{2}} \biggr]
\\
& \leq \frac{ ((4/3)\pi^{-1/2}+1 )g_m}{\sqrt{\log(d/s^*_d
- 1)}} \triangleq\tau' g_m=
\tau' 2^m,
\end{aligned}
\end{equation}
where the last inequality follows from the same argument as in \eqref{term18}. We denote by $Var_{\theta}\left( \llvert \hat\eta(g_{m}) - \eta \bigr
\rrvert \right)$ the variance of $\llvert \hat\eta(g_{m}) - \eta \bigr
\rrvert$. Observing that $\llvert \hat\eta(g_{m}) - \eta \bigr
\rrvert$ is a sum of independent Bernoulli random variables, we get
\begin{equation*}
\begin{aligned}
\mathbf{E}_\theta \bigl\llvert \hat\eta(g_{m}) - \eta \bigr
\rrvert^{2} & = Var_{\theta}\left( \llvert \hat\eta(g_{m}) - \eta \bigr
\rrvert \right) + \left( \mathbf{E}_\theta \bigl\llvert \hat\eta(g_{m}) - \eta \bigr
\rrvert\right)^{2}
\\
& \leq  \mathbf{E}_\theta \bigl\llvert \hat\eta(g_{m}) - \eta \bigr
\rrvert+ \left( \mathbf{E}_\theta \bigl\llvert \hat\eta(g_{m}) - \eta \bigr
\rrvert\right)^{2}.
\end{aligned}
\end{equation*}
Using \eqref{term241} and the fact that $\tau'$ is bounded, we get that
\begin{equation}\label{term243}
\mathbf{E}_\theta \bigl\llvert \hat\eta(g_{m}) - \eta \bigr
\rrvert^{2} \leq C_{2} \tau' 2^{2m},
\end{equation}
for some absolute constant $C_2>0$.

Now, we plug \eqref{term20} and \eqref{term243} in \eqref{term15} to obtain
%
\begin{equation*}
\begin{aligned}
I_2 &\le \frac{(C_2\tau')^{1/2}}s \sum_{m=m_0+1}^M
2^{m}\exp\bigl(-C_1 2^{m-1}\tau\bigr)
\\
&\le C_3 \bigl(\tau'\bigr)^{1/2}
\tau^{-1} \exp\bigl(-C_1 2^{m_0-1}\tau\bigr)\le
C_3 \bigl(\tau'\bigr)^{1/2}
\tau^{-1}
\end{aligned}
\end{equation*}
for some absolute constant $C_3>0$. Notice that $ (\tau')^{1/2} =
O ( (\log ({d}/{s^*_d} -1 ) )^{-\frac{1}{4}}
)$ as ${d}/{s^*_d}\to\infty$ while
$\tau^{-1}= O ( (\log ({d}/{s^*_d} -1 )
)^{\frac{1}{7}} )$. Thus, $I_2=o(1)$ as $d\to\infty$. Since from
\eqref{term14} we also get that $I_1=o(1)$ as $d\to\infty$, the
proof is complete.
\end{proof}

\begin{proof}[Proof of Lemma~\ref{lem1}]
Let first $s<g_M$. Then, by definition of $m_0$, we have $s<g_{m_0}$.
Therefore, $s<g_m$ for $m\ge m_0$, and we have $w(g_{m})<w(s)$. It
follows that
\[
a_0(s,A)- w(g_{m}) \ge a_0(s,A)- w(s) \ge
\frac{\sqrt{A}}{2\sqrt{2}} \min \biggl(\frac{\sqrt{A}}{\sqrt{2}}, \log^{1/4} ({d}/s -1
) \biggr),
\]
where we have used the elementary inequalities
\[
\sqrt{x+y}-\sqrt{x}\ge y/(2\sqrt{x+y}) \ge(2\sqrt{2})^{-1} \min (y/
\sqrt{x},\sqrt{y} )
\]
with $x=2\log ({d}/s -1 )$ and $y=A\sqrt{\log ({d}/s
-1 )}$. By assumption, $A\ge\break 16\sqrt{\log\log
({d}/{s^*_d} -1 )}$, so that we get
%
\begin{equation}
\label{term111} a_0(s,A)- w(g_{m}) \ge
a_0(s,A)- w(s) \ge{ 4} \biggl(\log\log \biggl(\frac{d}{s^*_d} -1
\biggr) \biggr)^{1/2}.
\end{equation}
This and the standard bound on the Gaussian tail probability imply
%
\begin{align}
 \mathbf{P}\bigl( \llvert \xi \rrvert >\bigl(a_0(s,A)-
w(g_{m})\bigr)_+\bigr) &\le \exp\bigl(-\bigl(a_0(s,A)-
w(g_{m})\bigr)^2/2\bigr)\nonumber
\\
&\le \bigl( \log \bigl({d}/{s^*_d} -1 \bigr) \bigr)^{-\frac{1}{2}}.\label{term13}
\end{align}
Let now $s\in[g_M, s^*_d]$. Then $m_0=M$ and we need to prove the
result only for $m=M$. By definition of $M$, we have $s^*_d\le2g_M$.
This and \eqref{term111} imply
\[
\begin{aligned}
a_0(s,A)- w(g_{M}) &\ge a_0(s,A)- w(s) -
\bigl(w\bigl(s^*_d/2\bigr)-w\bigl(s^*_d\bigr)\bigr)\\
& \ge{4}
\biggl(\log\log \biggl(\frac{d}{s^*_d} -1 \biggr) \biggr)^{1/2} -
\bigl(w\bigl(s^*_d/2\bigr)-w\bigl(s^*_d\bigr)\bigr).
\end{aligned}
\]
Now, using the elementary inequality $\sqrt{\log(x+y)}-\sqrt{\log
(x)}\le y/(2x\sqrt{\log(x)})$ with $x=d/s^*_d-1$ and $y=d/s^*_d$, and
the fact that $s^*_d\le d/4$ we find
\[
\begin{aligned}
w\bigl(s^*_d/2\bigr)-w\bigl(s^*_d\bigr) &\le
\frac{1}{\sqrt{2\log(d/s^*_d-1)}} \frac
{d}{d-s^*_d}\le\frac{2\sqrt{2}}{3\sqrt{\log(d/s^*_d-1)}}\\
&\le3 \biggl(\log\log
\biggl(\frac{d}{s^*_d} -1 \biggr) \biggr)^{1/2}.
\end{aligned}
\]
The last two displays yield $
a_0(s,A)- w(g_{M}) \ge (\log\log (\frac{d}{s^*_d}
-1 ) )^{1/2}$, and we conclude as in \eqref{term13}.
\end{proof}
\end{appendix}

\section*{Acknowledgments}
We would like to thank Felix Abramovich for helpful discussion of the results. The work of N.A. Stepanova was supported by an NSERC grant. The work of A.B. Tsybakov was supported by GENES and by the French National Research Agency (ANR) under the grants
IPANEMA (ANR-13-BSH1-0004-02), and Labex ECODEC (ANR - 11-LABEX-0047). It was also supported by the "Chaire Economie et Gestion des Nouvelles Donn\'ees", under the auspices of Institut Louis Bachelier, Havas-Media and Paris-Dauphine.

\bibliographystyle{plain}

\end{document}


\title{Supplement to "Variable selection with Hamming loss"}
\author{Butucea, C.$^{1,2}$, Ndaoud, M.$^{2}$, Stepanova, N.A.$^{3}$, and Tsybakov, A.B.$^{2}$\\
{\small $^1$ Universit\'e Paris-Est Marne-la-Vall\'ee, LAMA(UMR 8050), UPEM, UPEC, CNRS,}\\ {\small F-77454, Marne-la-Vall\'ee, France} \\
{\small $^2$ CREST, ENSAE, Universit\'e Paris-Saclay.
5, ave. Henry le Chatelier,} \\ 
{\small 91120 Palaiseau Cedex, France}\\
{\small $^3$ School of Mathematics and Statistics, Carleton University,
Ottawa, Ontario, K1S 5B6 Canada}
}

\maketitle
\begin{abstract}
\ We derive a general lower bound for the minimax risk over all selectors on the class of at most $s$-sparse vectors. The main term of this bound is a Bayes risk with arbitrary prior and the non-asymptotic remainder term is given explicitly. Using this,
we prove the lower bounds of Theorems~2.2, 3.2 and 3.3 in \cite{BNST}.
\end{abstract}






\section{Nonasymptotic lower bound of the minimax risk}

In the next theorem, we reduce the minimax risk over all selectors to a Bayes risk with arbitrary prior measure $\pi$ on $\{0,1\}^d$ and give a bound on the difference between the two risks. This result is true in a general setup, non necessarily for Gaussian models. For a particular choice of measure $\pi$,  we provide an explicit bound on the remainder term.

Consider the set of binary vectors
$$
\Theta_s = \left\{  \eta \in \{0,1\}^{d} : \  |\eta|_0 \leq s \right\}, \mbox{ where } |\eta|_0 = \sum_{j=1}^d I(\eta_j \not = 0),
$$
and assume that we are given a family $\{P_{\eta}, \eta \in \Theta_s\}$ where each $P_{\eta}$ is a probability distribution on a measurable space $(\mathcal{X}, \mathcal{U})$.
We observe $X$  drawn from $P_{\eta}$ with some unknown $\eta \in \Theta_s$ and we consider the Hamming risk of a selector $\hat{\eta}=\hat{\eta}(X)$:
$$
\sup_{ \eta \in \Theta_s}\mathbf{E}_{\eta}|\hat{\eta} - \eta|
$$
where $\mathbf{E}_{\eta}$ is the expectation with respect to $P_{\eta}$. Here and in what follows we denote by $|\eta - \eta'|$ the Hamming distance between two binary sequences $\eta, \eta' \in \{0,1\}^{d}$, and we call the selector any estimator with values in $\{0,1\}^{d}$. Let $\pi$ be a probability measure on $\{0,1\}^{d}$ (a prior on $\eta$). We denote by $\mathbb{E}_{\pi}$ the expectation with respect to~$\pi$.

\begin{theorem}\label{lowbound}
For any $s<d$ and any probability measure $\pi$ on $\{0,1\}^{d}$, we have
\begin{equation}\label{eeq1:th}
\inf_{\hat{\eta}} \sup_{ \eta \in \Theta_s} \mathbf{E}_\eta |\hat{\eta} - \eta|
\geq  \inf_{\hat T\in [0,1]^d}\mathbb{E}_{\pi}\mathbf{E}_\eta \sum_{j=1}^d | {\hat T}_j(X) - \eta_j | - 4\,\mathbb{E}_{\pi}\left[ |\eta|_0 I ( |\eta|_0 \ge  s+1 ) \right] ,
\end{equation}
where $\inf_{\hat{\eta}}$ is the infimum over all selectors and $\inf_{\hat T\in [0,1]^d}$ is the infimum over all  estimators $\hat T(X)=({\hat T}_1(X),\dots,{\hat T}_d(X))$ with values in $[0,1]^d$.

In particular, if $\pi$ is a product of $d$ Bernoulli distributions with parameters $d$ and $s'/d$ where $s'\in (0, s]$, we have
\begin{equation}\label{eeq2:th}
\inf_{\hat{\eta}} \sup_{ \eta \in \Theta_s} \mathbf{E}_\eta |\hat{\eta} - \eta|
\geq  \inf_{\hat T\in [0,1]^d}\mathbb{E}_{\pi}\mathbf{E}_\eta \sum_{j=1}^d | {\hat T}_j - \eta_j | - 4 s' \exp\Big(-\frac{(s-s')^2}{2s}\Big).
\end{equation}

\end{theorem}

\begin{proof}[Proof of Theorem~\ref{lowbound}]
Throughout the proof, we write for brevity $A=\Theta_s$.
Set $ \eta^{A} = \eta I(\eta \in A)$ and denote by $\pi_A$ the probability measure $\pi$ conditioned by the event $\{\eta \in A\}$,  that is, for any $C\subseteq \{0,1\}^{d}$,
$$
\pi_A(C) = \frac{\pi (C\cap \{\eta \in A\})}{\pi(\eta \in A)}\,.
$$
The measure $\pi_A$ is supported on $A$ and we have
\begin{eqnarray*}
\inf_{\hat{\eta}} \sup_{ \eta \in A} \mathbf{E}_\eta |\hat{\eta} - \eta|
&\ge&
 \inf_{\hat{\eta}} \mathbb{E}_{\pi_A}\mathbf{E}_\eta |\hat{\eta} - \eta|
= \inf_{\hat{\eta}} \mathbb{E}_{\pi_{A}}\mathbf{E}_\eta |\hat{\eta} - \eta^A|
\\
&\ge&   \sum_{j=1}^{d} \inf_{\hat{T_j}} \mathbb{E}_{\pi_A}\mathbf{E}_\eta |\hat{T}_j - \eta^A_j |
\end{eqnarray*}
where $\inf_{\hat T_j}$ is the infimum over all estimators $\hat T_j= \hat T_j(X)$ with values in $\mathbb{R}$.
According to
Theorem~1.1 and Corollary~1.2 on page 228 in \cite{LC}, there exists a Bayes estimator $B^A_j=B^A_j(X)$ such that
$$
\inf_{\hat{T_j}} \mathbb{E}_{\pi_A}\mathbf{E}_\eta |\hat{T}_j - \eta^A_j |=\mathbb{E}_{\pi_A}\mathbf{E}_\eta |B^A_j - \eta^A_j |,
$$
and this estimator is a conditional median of $\eta^A_j$ given $X$; in particular, for any estimator $\hat T_j(X)$ we have
\begin{equation}\label{eeq0}
\mathbb{E}^A\big(|B^A_j(X) - \eta^A_j | \big| X\big) \le \mathbb{E}^A\big(|\hat T_j(X) - \eta^A_j | \big| X\big)
\end{equation}
almost surely. Here, the superscript $A$ indicates that the conditional expectation $\mathbb{E}^A(\cdot|X)$ is taken when $\eta$ is distributed according to $\pi_A$.
Therefore,
\begin{equation}\label{eeq1}
\inf_{\hat{\eta}} \sup_{ \eta \in A} \mathbf{E}_\eta |\hat{\eta} - \eta|
\geq \mathbb{E}_{\pi_A}\mathbf{E}_\eta \sum_{j=1}^{d} | B^A_j - \eta^{A}_{j}| .
\end{equation}
Note that $B^A_j\in [0,1]$ since $\eta^A_j$ takes its values in $[0,1]$.
Using this, we obtain
\begin{align}\nonumber
  \inf_{ \hat{T}\in[0,1]^p} \mathbb{E}_{\pi}\mathbf{E}_\eta | \hat{T} - \eta | & \leq \mathbb{E}_{\pi}\mathbf{E}_\eta \sum_{j=1}^{d} | B^A_j - \eta_{j}| \\ \nonumber
    & = \mathbb{E}_{\pi}\mathbf{E}_\eta  \Big(\sum_{j=1}^{d} |B^A_j - \eta_{j}| I(\eta \in A) \Big) + \mathbb{E}_{\pi}\mathbf{E}_\eta
   \Big(\sum_{j=1}^{d} |B^A_j - \eta_{j}|  I(\eta \in A^c) \Big)
  \\  \nonumber
  & = \mathbb{E}_{\pi_A}\mathbf{E}_\eta  \sum_{j=1}^{d} |B^A_j - \eta_{j}^A|  + \mathbb{E}_{\pi}\mathbf{E}_\eta
   \Big(\sum_{j=1}^{d} |B^A_j - \eta_{j}|  I(\eta \in A^c)  \Big)
   \\  \label{eeq2}
  & \le \mathbb{E}_{\pi_A}\mathbf{E}_\eta  \sum_{j=1}^{d} |B^A_j - \eta_{j}^A|  +
  \mathbb{E}_{\pi}\mathbf{E}_\eta
   \sum_{j=1}^{d} B^A_j  I(\eta \in A^c)
   +
   \mathbb{E}_{\pi}
   \sum_{j=1}^{d} \eta_j I (\eta \in A^c) .
 \end{align}
Our next step is to bound the term
$$
\mathbb{E}_{\pi}\mathbf{E}_\eta
   \sum_{j=1}^{d} B^A_j I(\eta \in A^c).
$$
For this purpose, we first note that inequality \eqref{eeq0} with $\hat T_j(X) = \mathbb{E}^A(\eta_j^A|X)$ implies that
$$
B^A_j (X) = \mathbb{E}^A(B_j^A(X)|X) \le \mathbb{E}^A\big(|\mathbb{E}^A(\eta_j^A|X)| \big| X\big) + 2\mathbb{E}^A\big(|\eta_j^A|\big|X\big) = 3\mathbb{E}^A(\eta_j^A|X)
$$
where we have used the fact that $\eta_j^A\in [0,1]$. Since $\sum_{j=1}^{d}\eta_{j}^{A} \leq s $ (cf. definition of $\eta_j^A$),  we find that
$
\sum_{j=1}^{d} B^A_j \le 3s.
$
Finally, as $\sum_{j=1}^{d} \eta_j>s$ on $A^c$ we get $\sum_{j=1}^{d} B^A_jI(\eta \in A^c) \le 3\sum_{j=1}^{d} \eta_j I(\eta \in A^c)$, and thus
$$
\mathbb{E}_{\pi}\mathbf{E}_\eta
    \sum_{j=1}^{d} B^A_j I(\eta \in A^c)\le
   3\,\mathbb{E}_{\pi}
    \sum_{j=1}^{d} \eta_j I(\eta \in A^c).
$$
Combining this inequality with \eqref{eeq1} and \eqref{eeq2} yields \eqref{eeq1:th}.

\bigskip

We now prove inequality \eqref{eeq2:th}. In this case, $\sum_{j=1}^{d} \eta_j:= \zeta$ has the binomial distribution~${\mathcal B}(d,q)$ with parameters $d$ and $q=s'/d$. Then,
\begin{align*}
\mathbf{E}(\zeta I(\zeta \ge s+1))& = \sum_{k=s+1}^{d} k {d\choose k} q^k(1-q)^{d-k}\\
& = \sum_{k=s+1}^{d} \frac{d(d-1)!}{(k-1)!(d-k)!} q^k(1-q)^{d-k}\\
& = dq \sum_{k=s+1}^{d} {d-1\choose k-1} q^{k-1}(1-q)^{d-k}\\
& = dq \sum_{m=s}^{d-1} {d-1\choose m} q^{m}(1-q)^{(d-1)-m}\\
&= s' \mathbf{P}({\mathcal B}(d-1,s'/d)\ge s) \le s' \mathbf{P}({\mathcal B}(d,s'/d)\ge s).
\end{align*}
Thus, to complete the proof of  \eqref{eeq2:th}, it is enough to bound the probability $\mathbf{P}({\mathcal B}(d,s'/d)\ge s)$. To this end, we use the following lemma, which is a  combination of formulas (3)  and (10) on pages 440--441 in~\cite{SW}.

\begin{lemma}\label{binomial}
	Let $\mathcal{B}(d,q)$ be the binomial random variable with parameters $d$ and~$q\in (0,1)$.
	Then, for any $\lambda>0$,
	\begin{align}\label{binomial1}
		&\mathbf{P}\big(\mathcal{B}(d,q)\ge\lambda \sqrt{d}+dq\big) \le \exp\bigg(-\frac{\lambda^{2}}{2q(1-q)\big(1+\frac{\lambda}{3q\sqrt{d}}\big)}\bigg).
		\end{align}
		\end{lemma}
Applying this lemma with $q=s'/d$ and $\lambda=(s-s')/\sqrt{d}$ we find that
$$
\mathbf{P}({\mathcal B}(d,s'/d)\ge s)\le \exp\bigg(-\frac{(s-s')^{2}}{2s}\bigg).
$$
Thus, \eqref{eeq2:th} follows.

\end{proof}

{\bf Remark.} If we take $s'= s-(3/2)\sqrt{s\log(s)}$ (which is possible since for all $s\ge 2$ we have $s'>0$) inequality \eqref{eeq2:th} implies that, for all $s\ge2$,
\begin{equation*}\label{eeq3:th}
\inf_{\hat{\eta}} \sup_{ \eta \in A_s} \mathbf{E}_\eta |\hat{\eta} - \eta|
\geq  \inf_{\hat T\in [0,1]^d}\mathbb{E}_{\pi}\mathbf{E}_\eta \sum_{j=1}^d | {\hat T}_j - \eta_j | - 4 s^{-1/8}.
\end{equation*}


\section{Proofs of some results in \cite{BNST}}

\begin{proof}[Proof of the second lower bound in Theorem~2.2]
This bound follows directly from the lower bounds of Theorems~3.2 and~3.3 that hold
for general distributions.
\end{proof}

\begin{proof}[Proof of Theorem~3.2]
The upper bound $\sup_{\eta \in \Theta_s
}\E_\eta |\hat \eta - \eta| \le \Psi(d,s) sd/(d-s)$ is straightforward in view of the definition of $\hat \eta$.

We now 
prove the lower bound of Theorem~3.2.
%
%
First note that, in view of \eqref{eeq2:th}, the proof is reduced  to showing  that
\begin{equation}\label{eq:sep}
    \inf_{\hat T\in [0,1]^d}\mathbb{E}_{\pi}\mathbf{E}_\eta \sum_{j=1}^d | {\hat T}_j - \eta_j | \geq s'\Psi(d,s),
\end{equation}
where $\pi$ is a product on $d$ Bernoulli distributions with parameter $s'/d$ and $s'\in (0, s]$.
We have 
\begin{eqnarray*}\inf_{\hat T\in [0,1]^d}\mathbb{E}_{\pi}\mathbf{E}_\eta \sum_{j=1}^d | {\hat T}_j - \eta_j | & \geq &\sum_{j=1}^d \inf_{\hat T_{j}\in [0,1]}\mathbb{E}_{\pi}\mathbf{E}_\eta  | {\hat T}_j(X) - \eta_j |
\\ 
&\ge &\sum_{j=1}^d \mathbb{E}_{\pi}\mathbf{E}_\eta \Big( \inf_{\hat T_{j}\in [0,1]} \mathbb{E}_{\pi}\mathbf{E}_\eta  \big(| {\hat T}_j(X) - \eta_j | \, \big| \{(X_j, \eta_j), j\ne i\}\big)\Big).
\end{eqnarray*}
Since the components $X_j$ of $X=(X_1,\dots,X_d)$ are independent, the $j$th  conditional expectation in the last expression reduces to the unconditional expectation over $(X_j,\eta_j)$, which is bounded from below by 
$$
\inf_{T\in [0,1]}\Big(\Big(1-\frac{s'}{d}\Big)E_0 (T) + \frac{s'}{d}E_1 (1-T)\Big) =  s'\Psi(d,s')/d.
$$
Here, $E_i$ is the expectation with respect to $P_i$, $i=0,1$. Thus,
$$ \inf_{\hat T\in [0,1]^d}\mathbb{E}_{\pi}\mathbf{E}_\eta \sum_{j=1}^d | {\hat T}_j - \eta_j | \geq s'\Psi(d,s'). $$
To finish the proof of \eqref{eq:sep}, it remains to show that
$\Psi(d,s') \geq \Psi(d,s)$ for all $s'\leq s$. For this purpose, we extend $\Psi(d,\cdot)$ to $\mathbb{R}_+$ by defining, for all $u>0$,
 \begin{eqnarray*}
 \Psi(d,u) &=& P_{1}\left( uf_{1}(X_1) - (d-u)f_{0}(X_1) < 0 \right) + \left(\frac{d}{u}-1\right)P_{0}\left( uf_{1}(X_1) - (d-u)f_{0}(X_1) \geq 0 \right)\\
 &=& 1- E_1 [ I(Y(u) \geq 0)] + \left( \frac{d}{u}-1 \right) E_0[ I(Y(u) \geq 0)],
 \end{eqnarray*}
 where $Y(u) = uf_{1}(X_1) - (d-u)f_{0}(X_1)$.
For $\epsilon>0$, we define the function $g_{\epsilon}:\mathbb{R} \to \mathbb{R}_+$ by
$$ g_{\epsilon}(u) =  \frac{2u^{2}}{\epsilon^{2}} I\left( 0 \leq u < \frac{\epsilon}2 \right)+ \left( 1- \frac{2(\epsilon - u)^{2}}{\epsilon^{2}} \right) I\left( \frac{\epsilon}2 \leq u < \epsilon \right) + I\left( u \geq \epsilon \right) .  $$
It is easy to check that  $g_{\epsilon}$ is continuously differentiable on $\mathbb{R}$ and that, for all $u$ in $\mathbb{R}$,
$$
\underset{\epsilon \to 0}{\lim}\, g_{\epsilon}(u) = I( u \geq 0 ) \mbox{ and }
u g_{\epsilon}'(u) \geq 0.
$$
Finally, for $\epsilon>0$  and $u>0$, we define $\Psi_{\epsilon}(d,u)$ by the formula
$$
\Psi_{\epsilon}(d,u) = 1 - E_1\left[ g_{\epsilon}(Y(u) )   \right] + \left( \frac{d}{u}-1 \right) E_0\left[ g_{\epsilon}(Y(u)) \right] .  $$
An application of the dominated convergence theorem proves that $ \underset{\epsilon \to 0}{\lim}\, \Psi_{\epsilon}(d,u) = \Psi(d,u)$ for all
$u >0$, and that one can differentiate in the expression for $\Psi_{\epsilon}(d,\cdot)$ under the expectation signs on $\mathbb{R}_+$.  We also note that $\Psi_{\epsilon}(d,\cdot)$ is decreasing on $\mathbb{R}_{+}$. Indeed, for any $u>0$,
\begin{align*}
    \frac{\partial }{\partial u}\Psi_{\epsilon}(d,u) &= -\frac{d}{ u^{2}}E_0\left[ g_{\epsilon}(uf_{1}(X_1) - (d-u)f_{0}(X_1)) \right] \\
    &- \sum_{i=0,1} E_{i} \Big[g_{\epsilon}'(uf_{1}(X_1) - (d-u)f_{0}(X_1)) \frac{ uf_{1}(X_1) - (d-u)f_{0}(X_1)}{u} \Big].
\end{align*}
Using the inequality $ w\,g_{\epsilon}'(w) \geq 0$ we get that $\Psi_{\epsilon}(d,\cdot)$ is decreasing on $\mathbb{R}_{+}$.
Finally, pointwise convergence of $\Psi_{\epsilon}(d,\cdot )$ to $\Psi(d,\cdot )$ implies that $\Psi(d,\cdot )$ is also decreasing on $\mathbb{R}_{+}$.
\end{proof}

\begin{proof}[Proof of Theorem~3.3]
We have
\begin{eqnarray*}
\sup_{\theta \in \Theta_d^+(s, a_0,a_1)}\frac 1s \E_\theta |\hat \eta - \eta|
&=& \sup_{a \geq a_1} \frac{|S|}s \Pb_a \Big(\log \frac{f_1}{f_0}(X_1) < \log \Big(\frac ds - 1\Big)\Big)\\
&& + \sup_{a\leq  a_0} \Big(\frac{d-|S|}s\Big)  \Pb_a \Big(\log \frac{f_1}{f_0}(X_1) \geq \log \Big(\frac ds - 1\Big)\Big) \\
&=& \frac{|S|}s \Pb_{a_1} \Big(\log \frac{f_1}{f_0}(X_1) < \log \Big(\frac ds - 1\Big)\Big)\\
&& + \Big(\frac{d-|S|}s\Big)  \Pb_{a_0} \Big(\log \frac{f_1}{f_0}(X_1) \geq \log \Big(\frac ds - 1\Big)\Big)\\
& \leq & \Psi(d,s) \frac d{d-s},
\end{eqnarray*}
where the last equality is due to the monotonicity of $\log \frac{f_1}{f_0} (X)$ and to the stochastic order of the family $\{{\sf f}_{a}, \, a \in \mathcal{U}\}$.

The lower bound on the minimax risk $$\inf_{\tilde \eta} \sup_{\theta \in \Theta_d^+(s, a_0,a_1)} \frac 1s \E_\theta|\tilde \eta - \eta|$$
follows from the lower bound of Theorem~3.2 by taking there $f_0={\sf f}_{a_0}$ and $f_1={\sf f}_{a_1}$. \end{proof}


